\documentclass[a4paper,12pt,reqno]{amsart}
\usepackage[margin=3cm]{geometry}
\usepackage[T1]{fontenc}
\usepackage{soul}
\usepackage{amsmath,amssymb,amsfonts,amsthm}
\usepackage{xcolor}

\definecolor{darkgreen}{rgb}{0.00,0.33,0.25}
\definecolor{darkred}{rgb}{0.60,0.05,0.05}
\definecolor{darkblue}{rgb}{0.05,0.05,0.60}
\numberwithin{equation}{section}

\usepackage[colorlinks = true,
     linkcolor = darkgreen,
     anchorcolor = darkgreen,
     citecolor = darkblue,
     filecolor = darkgreen,
     urlcolor = darkgreen
]{hyperref}
\usepackage[english]{babel}  
\addto\extrasenglish{  
    
}

\usepackage{mathtools,accents}
\usepackage{mathrsfs}
\usepackage{aliascnt}
\usepackage{braket}
\usepackage{bm}
\usepackage{todonotes}
\usepackage{esint}

\usepackage{enumerate}

\makeatletter
\newcommand*\rel@kern[1]{\kern#1\dimexpr\macc@kerna}
\newcommand*\widebar[1]{%
  \begingroup
  \def\mathaccent##1##2{%
    \rel@kern{0.8}%
    \overline{\rel@kern{-0.8}\macc@nucleus\rel@kern{0.2}}%
    \rel@kern{-0.2}%
  }%
  \macc@depth\@ne
  \let\math@bgroup\@empty \let\math@egroup\macc@set@skewchar
  \mathsurround\z@ \frozen@everymath{\mathgroup\macc@group\relax}%
  \macc@set@skewchar\relax
  \let\mathaccentV\macc@nested@a
  \macc@nested@a\relax111{#1}%
  \endgroup
}
\makeatother

\makeatletter
\def\newaliasedtheorem#1[#2]#3{
  \newaliascnt{#1@alt}{#2}
  \newtheorem{#1}[#1@alt]{#3}
  \expandafter\newcommand\csname #1@altname\endcsname{#3}
}
\makeatother

\usepackage{aliascnt}
\def\NewTheorem#1{%
	  \newaliascnt{#1}{theorem}
	  \newtheorem{#1}[#1]{\expandafter\MakeUppercase #1}
	  \aliascntresetthe{#1}
	  \expandafter\def\csname #1autorefname\endcsname{\MakeUppercase #1}
}

\theoremstyle{plain}
	\newtheorem{theorem}{Theorem}[section]
	\NewTheorem{corollary}
	\NewTheorem{lemma}  
	\NewTheorem{proposition}
	\NewTheorem{conjecture}
	\NewTheorem{hypothesis}
	\NewTheorem{definition}
	\NewTheorem{assumption}
	\NewTheorem{notation}
	\newtheorem*{definition*}{Definition}

\theoremstyle{remark}
	\NewTheorem{remark}
	\NewTheorem{example}
	
	\newtheorem*{claim*}{Claim}
	\newtheorem*{remark*}{Remark}
	\newtheorem*{example*}{Example}
	\newtheorem*{notation*}{Notation}

\DeclareMathOperator{\diam}{diam}
\DeclareMathOperator{\Ent}{Ent}

\newcommand{\R}{\mathbb{R}}

\newcommand{\N}{\mathbb{N}}
\newcommand{\dd}{\, \mathrm{d}}
\newcommand{\ddd}{\mathrm{d}}

\newcommand{\PCN}{\mathsf{PC}_N}
\newcommand{\PCcT}{\mathsf{PC}_\cT}

\newcommand{\weakto}{\rightharpoonup}

\DeclareMathOperator{\Lip}{Lip}

\newcommand{\Haus}{\mathscr{H}}
\newcommand{\Leb}{\mathscr{L}}

\newcommand{\Q}{\text{Q}}

\newcommand{\suchthat}{\ensuremath{\ : \ }} 
\DeclareMathOperator{\supp}{supp}
\newcommand{\abs}[1]{\ensuremath{\left\lvert #1 \right\rvert}} % absolute value
\newcommand{\cD}{\mathcal{D}}
\newcommand{\sH}{\mathscr{H}}
\newcommand{\cP}{\mathscr{P}}
\newcommand{\cL}{\mathcal{L}}
\newcommand{\cX}{\mathcal{X}}

\newcommand{\cE}{\mathcal{E}}
\newcommand{\cH}{\mathcal{H}}
\newcommand{\cA}{\mathcal{A}}
\newcommand{\cU}{\mathcal{U}}
\newcommand{\cM}{\mathcal{M}}

\newcommand{\cW}{\mathcal{W}}
\newcommand{\cT}{\mathcal{T}}
\newcommand{\cF}{\mathcal{F}}
\newcommand{\cI}{\mathcal{I}}

\newcommand{\cO}{\mathcal{O}}
\newcommand{\cS}{\mathcal{S}}

\newcommand{\bH}{\mathbf{H}}
\newcommand{\bE}{\mathbf{E}}
\newcommand{\bA}{\mathbf{A}}

\newcommand{\bI}{\mathbf{I}}
\newcommand{\bF}{\mathbf{F}}
\newcommand{\bM}{\mathbf{M}}
\newcommand{\tbF}{\tilde\bF}

\newcommand{\bW}{\mathbf{W}_2}

\newcommand{\ssubset}{\Subset}

\newcommand{\vphi}{\varphi}

\newcommand{\Cyl}{\text{Cyl}}
\newcommand{\simh}{\stackrel{h}{\sim}}

\newcommand{\bOmega}{\overline\Omega}

\newcommand{\fm}{\mathfrak{m}}
\newcommand{{\bfm}}{\widebar{\mathfrak{m}}}

\newcommand{\eps}{\varepsilon}

\def\avint{\mathop{\,\rlap{--}\!\!\int}\nolimits}
\newcommand{\tand}{\quad\text{ and }\quad}
\newcommand{\ip}[1]{\langle {#1}\rangle}

\newcommand{\ddt}{\frac{\mathrm{d}}{\mathrm{d}t}}

\def\supp{\operatorname{supp}}

\def\tand{\quad{\rm and}\quad}

\definecolor{jan}{rgb}{0.0,0.3,0.8}
\definecolor{mat}{rgb}{0.0,0.5,0.3}
\definecolor{ube}{rgb}{0.72, 0.52, 0.91}

\author{Dominik Forkert} 
\author{Jan Maas}
\author{Lorenzo Portinale}

\address{Institute of Science and Technology Austria (IST Austria),
Am Campus 1, 3400 Klosterneuburg, Austria}
\email{dominik.forkert@ist.ac.at}
\email{jan.maas@ist.ac.at}
\email{lorenzo.portinale@ist.ac.at}

\title[Evolutionary $\Gamma$-convergence for Fokker-Planck equations]{Evolutionary $\Gamma$-convergence of entropic gradient flow structures for Fokker-Planck equations in multiple dimensions}

\begin{document}

\setcounter{tocdepth}{1}

\begin{abstract}
We consider finite-volume approximations of Fokker-Planck equations on bounded convex domains in $\R^d$ and study the corresponding gradient flow structures. 
We reprove the convergence of the discrete to continuous Fokker-Planck equation via the method of Evolutionary $\Gamma$-convergence, i.e., we pass to the limit at the level of the gradient flow structures, generalising the one-dimensional result obtained by Disser and Liero. 
The proof is of variational nature and relies on a Mosco convergence result for functionals in the discrete-to-continuum limit that is of independent interest.
Our results apply to arbitrary regular meshes, even though the associated discrete transport distances may fail to converge to the Wasserstein distance in this generality.
\end{abstract}

\maketitle

\section{Introduction}
This paper deals with the convergence of discrete gradient flow structures arising from finite volume discretisations of Fokker-Planck equations on bounded convex domains $\Omega \subset \R^d$. 
For a given potential $V \in C(\bOmega)$ we consider the Fokker-Planck equation
\begin{align}	\label{eq:intro_FP}
	\partial_t \mu = \Delta \mu + \nabla \cdot ( \mu \nabla V ) \quad \text{on }(0,T) \times \Omega, \quad \mu|_{t=0} = \mu_0
\end{align} 
with no-flux boundary conditions, for $T \in (0,+\infty)$.  
Since the seminal works of Jordan, Kinderlehrer, and Otto \cite{jordan1998,O01} it is known that \eqref{eq:intro_FP} can be formulated as a gradient flow in the space of probability measures $\cP(\bOmega)$ endowed with the $2$-Wasserstein distance $\bW$ from optimal transport. 
The driving functional is the relative entropy with respect to the invariant measure
$\bfm(\ddd x) := \frac{1}{Z_V} \exp(-V(x)) \dd x$, where $Z_V$ is a normalising constant. 
Here we consider spatial discretisations of \eqref{eq:intro_FP} obtained by finite volume methods for a general class of admissible meshes. 
In this setting it is very well known that solutions to the discrete equations converge to solutions
of \eqref{eq:intro_FP}; see, e.g., \cite{eymard2000finite}, \cite{barth2018finite} for results in dimension $2$ and $3$, and \cite{droniou2018gradient} for more general situations.

The discretised Fokker-Planck equation 
can also be formulated as gradient flow, with respect to a suitable discrete dynamical transport distance $\cW_\cT$; see the independent works \cite{chow12,maas2011gradient,mielke2011gradient}.
Here we exploit this gradient flow structure to reprove the convergence of discrete to continuous Fokker-Planck equations via the method of \textit{evolutionary $\Gamma$-convergence}; i.e., rather than directly passing to the limit at the level of the gradient flow equation, we pass to the limit in the \emph{energy-dissipation inequality} that characterises the gradient flow structure.

This yields a new proof of convergence for the associated gradient flow equations, which does not rely on specific properties such as linearity or second order. 
Instead, the method is based on properties of functionals and tools such as $\Gamma$- and Mosco convergence.

The method of evolutionary $\Gamma$-convergence was pioneered by Sandier and Serfaty \cite{sandier04}; see \cite{mielke2016} for a survey on the topic and \cite{mielke2020exploring} for important refinements.
It has recently been applied to gradient system with a wiggly energy \cite{dondl2019gradient,mielke2020exploring}, coarse graining in linear fast-slow reaction systems \cite{mielke2019coarse}, diffusion in thin structures \cite{frenzel2018effective}, chemical reaction systems \cite{maas2020modeling}, and various other situations.

For Fokker-Planck equations in dimension $d = 1$, evolutionary $\Gamma$-convergence of the discrete gradient flow structures was proved by Disser and Liero \cite{disser2015}, for a specific class of finite-volume discretisations (cf. Section \ref{section:previous_works}). 
Their proof relies on interpolation techniques which do not easily extend to multiple dimensions.
Our proof is different and relies on compactness and representation theorems, in particular \cite[Theorem 2]{Bouchitte2002}, adapting ideas from \cite{Alicandro2004}.  
Our variational proof suggests the possibility of extending these techniques to more general settings, e.g., to higher order and/or nonlinear PDEs.

The fact that the method of evolutionary $\Gamma$-convergence of the gradient structures works on arbitrary admissible meshes is remarkable in view of recent work on the discrete-to-continuous limit of the associated transport distances. 
In fact, it was shown in \cite{gladbach2018} that the convergence of the discrete transport distances to the Wasserstein distance $\bW$ (in the limit of vanishing mesh size) requires a restrictive isotropy condition on the meshes; see \cite{gladbach2019} for explicit examples.
This discrepancy in convergence behaviour can be explained by a difference in regularity: 
to prove $\Gamma$-convergence of the discrete gradient flow structures one may exploit spatial smoothness assumptions on the discrete dynamics 
(in view of regularity results for the discrete gradient flows). 
By contrast, the transport costs on anisotropic meshes are minimised along highly oscillatory curves.

\vspace{3mm}
\noindent
\textbf{Organisation of the paper}.
In Section \ref{sec:gradient_flows} we discuss  gradient flow structures for continuous and discretised Fokker-Planck equations.
Section \ref{section:statements} contains the main result of this paper, namely, the evolutionary $\Gamma$-convergence of discrete to continuous gradient flow structures (Theorem \ref{thm:MAIN}).
This result relies on energy bounds (Theorem \ref{thm:lower_bounds}) which are proved using Mosco convergence results in the discrete-to-continuum limit that are of independent interest (Theorem \ref{thm:Mosco}). 
In Section \ref{section:previous_works} we discuss related work.
Section \ref{section:evolut_gamma_Wass} contains the proofs of Theorem \ref{thm:lower_bounds} and Theorem \ref{thm:MAIN}. 
The proof of Theorem \ref{thm:Mosco} is contained in Sections \ref{section:mosco},  \ref{section:moscop}, and \ref{section:compactness_representation}.

\section{Finite-volume discretisation of Wasserstein gradient flows}
\label{sec:gradient_flows}
In this section we describe the formulation of the Fokker-Planck equations as gradient flow in the space of  probability measures, both at the continuous and at the discrete level.
For the sake of clarity, our discussion will be informal. 
We refer to Section~\ref{section:statements} below for rigorous statements of the main results.

\subsection{Fokker-Planck equations as Wasserstein gradient flows}
On a bounded convex domain $\Omega \subset \R^d$ we consider the Fokker-Planck equation
\begin{align}	\label{eq:FP0}
	\partial_t \mu_t = \Delta \mu_t + \nabla \cdot ( \mu_t \nabla V )
\end{align}
with no-flux boundary conditions, where $V \in  C(\bOmega) \cap C^1(\Omega)$ is a driving potential.
This equation describes the time-evolution of the law of a Brownian particle in a potential field.
The steady state is given by the probability measure 
\begin{equation}		\label{eq:steady-state1}
\bfm \in \cP(\bOmega) \quad\text{with density}\quad \sigma(x) = \frac{\ddd \bfm}{\ddd x} = \frac{1}{Z_V} e^{-V(x)},
\end{equation}
where  $Z_V \in \R_+$ is a normalising constant. 

Since the seminal work of Jordan, Kinderlehrer and Otto \cite{jordan1998} it is known that \eqref{eq:FP0} is a gradient flow with respect to the Wasserstein distance $\bW$ from optimal transport.
In its dynamical formulation, $\bW$
is given by the \emph{Benamou--Brenier formula}
\begin{equation}	\label{eq:def W2 dynamical}
	\bW(\mu_0, \mu_1)^2 
	= \inf\bigg\{
		\int_0^1 
			\int_{\bOmega} |v_t(x)|^2 \dd \mu_t(x)
		\dd t \bigg\}, 
\end{equation}
where the infimum runs over all curves $(\mu_t)_t$ in the space of probability measures and all vector fields $(v_t)_t$ satisfying the continuity equation
\begin{align*}
	\label{eq:continuity equation cont}
	\partial_t \mu_t + \nabla \cdot (\mu_t v_t) = 0
\end{align*}
in the sense of distributions, with boundary conditions $\mu_t|_{t = 0} = \mu_0$ and $\mu_t|_{t = 1} = \mu_1$.
The driving functional in this gradient flow formulation is the relative entropy $\bH : \cP(\bOmega) \to [0, + \infty]$ given by
\begin{equation*}
\bH(\mu) := \begin{cases}
\int_{\bOmega} \rho(x) \log \rho(x) \dd {\bfm} & \text{if }\dd \mu = \rho \dd {\bfm},\\
+ \infty & \text{otherwise.}
\end{cases}
\end{equation*}
The gradient flow structure can be interpreted at various levels: the original formulation in \cite{jordan1998} was given in terms of a time-discrete minimising movement scheme. 
Another interpretation is in terms of Otto's formal infinite-dimensional Riemannian calculus on the Wasserstein space \cite{O01}.
Yet another approach relies on the metric formulation of gradient flows in terms of the
\emph{energy dissipation inequality} (EDI)
\begin{equation}
\label{eq:EDI}
	\bH(\mu_t) + \frac{1}{2} \int_0^T |\dot{\mu}_t|_{\bW}^2 + | \partial_{\bW} \bH (\mu_t)|^2 \dd t \leq \bH(\mu_0),
\end{equation}
where $|\dot{\mu}_t|$ denotes the $\bW$-metric derivative of the curve $\mu_t$ and $\partial_{\bW} \bH(\mu)$ the slope of the relative entropy, namely
\begin{align*}
	|\dot{\mu}_t|_{\bW}:= \lim_{h \rightarrow 0} \frac{1}{h} \bW(\mu_{t+h}, \mu_t), \qquad
	|\partial_{\bW} \bH (\mu)| := \limsup_{\nu \rightarrow \mu} \frac{[\bH(\mu) - \bH(\nu)]_{-}}{\bW(\mu,\nu)},
\end{align*}
where $[a]_- = \max\{0, -a\}$.
Writing $\rho = \frac{\ddd \mu}{\ddd \bfm}$, we have the identity
\begin{align}	\label{eq:slope_representations}
	|\partial_{\bW} \bH (\mu)|^2
	 = \bI(\mu), \quad 
	 \text{ where } \quad
	 \bI(\mu) := \int_{\Omega}
	 	 | \nabla \log \rho |^2	\rho 
	   \dd \bfm 
	 = 4\int_{\Omega}
	 	 | \nabla \sqrt{\rho} |^2  
	   \dd {\bfm}
\end{align}
is the \textit{relative Fisher information} with respect to $\bfm$.

\subsection*{$\bA$-$\bA^*$ formalism of gradient flows}

One can recast \eqref{eq:EDI} in terms of a suitable weighted Dirichlet energy $\bA$ and its Legendre transform $\bA^*$. 
Let us consider the energy functional
\begin{equation}	\label{eq:defA}
	\bA(\mu, \vphi) 
		:= \frac{1}{2} \int_{\bOmega} 
				| \nabla \vphi |^2 \dd \mu, 
		\quad \vphi \in C_c^{\infty}(\R^d), \; 
	\mu \in \cP(\bOmega),
\end{equation}
and its Legendre dual of $\bA$ with respect to the second variable:
\begin{align*}
	\bA^*(\mu, \eta) 
		= \sup_{\vphi \in C_c^{\infty}(\R^d)}
			 \{ \ip{ \vphi, \eta } - \bA(\mu, \vphi) \}
\end{align*}
for any distribution $\eta \in \cD'(\R^d)$. Note that
$
	\bA^*(\mu, w) 
		= \bA(\mu, \vphi) 
$
whenever $w = -\nabla \cdot ( \mu \nabla \vphi)$.
 The connection to Wasserstein geometry is given by the infinitesimal Benamou--Brenier formula
\begin{align*}
	\frac{1}{2}|\dot{\mu}_t|_{\bW}^2 
	= \bA^*(\mu_t, \partial_t \mu_t).
\end{align*}
Moreover, the relative Fisher information can be written as 
\begin{align}	\label{eq:Fisher_Onsager}
	\bI(\mu) = 2 \bA\big( \mu,-D\bH(\mu) \big),
\end{align}
where $D \bH(\mu) = \log \rho$ is the $L^2(\bfm)$-differential of $\bH$. 
Hence, it follows that \eqref{eq:EDI} can be stated equivalently as
\begin{equation}
\label{eq:contEDI}
	\bH(\mu_T) +  \int_0^T \bA^*(\mu_t, \dot{\mu}_t) + \bA\big(\mu_t, -D \bH(\mu_t)\big) \dd t \leq \bH(\mu_0).
\end{equation} 
This formulation is particularly convenient to relate the discrete framework to the continuous one, as we discuss in the next subsection.

\subsection{The discrete Fokker-Planck equation as gradient flow}

We consider a finite volume discretisation of $\bOmega$, closely following the setup in \cite{eymard2000finite}.
We thus consider finite partition $\cT$ of $\bOmega$ into sets (called cells) with nonempty and convex interior. 
Note that all interior cells are polytopes. 
We assume that $\cT$ is \emph{admissible}, in the sense that each of the cells $K \in \cT$ contains a point $x_K \in \overline{K}$ 
such that $x_K - x_L$ is orthogonal to the boundary surface $\Gamma_{KL} := \partial \overline{K} \cap \partial \overline{L}$, whenever $K$ and $L$ are \emph{neighbouring cells}, i.e., $\Haus^{d-1}(\Gamma_{KL}) > 0$. 
In this case we write $K \sim L$.
This  in a standard finite-volume setup.

\begin{figure}[h]
    \begin{center}
		\includegraphics[scale=0.4]{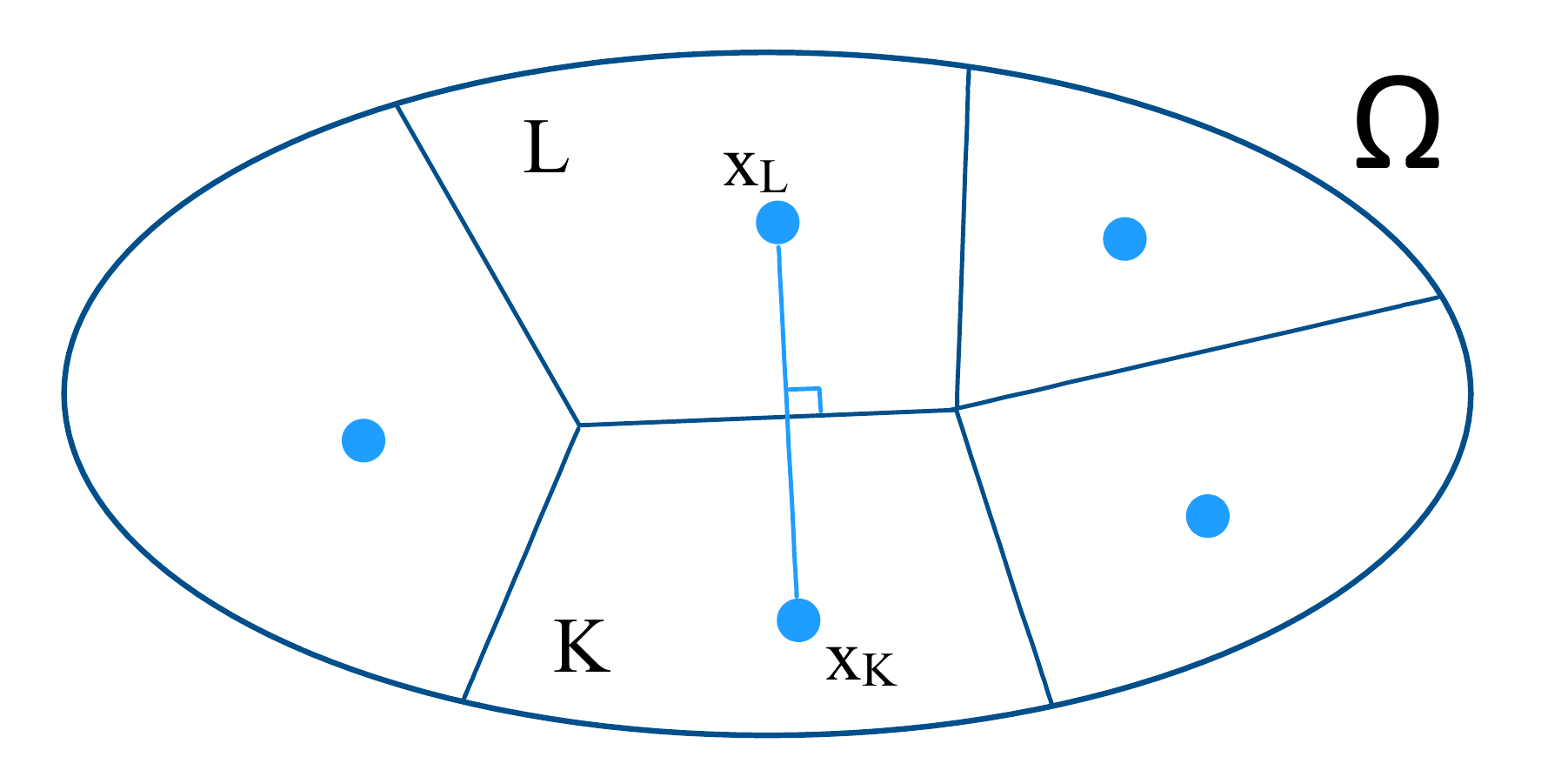}
    \end{center}
    \caption{An admissible mesh with cells $K, L, \ldots$ on a domain $\Omega \subset \R^d$.}
\label{fig:1D}
\end{figure}
An admissible mesh is said to be \emph{$\zeta$-regular} for some $\zeta \in (0,1]$, if the following conditions hold:
\begin{equation*}\begin{aligned}
	& (\textit{inner ball}) 
	& B\big(x_K, \zeta [\cT]\big) 
	& \subseteq K
	&   \text{ for all } K \in \cT, \\ 
    & (\textit{area bound}) 
    & \Haus^{d-1}(\Gamma_{KL}) 
    & \geq \zeta [\cT]^{d-1}
	& \text{ for all } K, L \in \cT \text{ with } K \sim L,
\end{aligned}\end{equation*}
where $[\cT] := \max \big\{ \diam(K) \ : \  K \in \cT \big\}$ denotes the size of the mesh. 

\subsection*{Discrete Fokker-Plank equations}

We consider discrete Fokker-Planck equations of the form 
\begin{align}
	\label{eq:discrete-FP}
		\ddt m_t(K)
			= \sum_{L \sim K} w_{KL} \bigg(
				\frac{m_t(L)}{\pi_\cT(L)}  
			-	\frac{m_t(K)}{\pi_\cT(K)} 
				\bigg).
\end{align}
Here, the probability measure $\pi_\cT \in \cP(\cT)$ is the canonical discretisation of $\bfm$, and the coefficients $w_{KL}$ are defined using the geometry of the mesh:
\begin{align}	\label{eq:defw}
	\pi_\cT(\{K\})
		:=  \bfm (K), \qquad
	w_{KL} 
		:= 
		\frac{\Haus^{d-1} ( \Gamma_{KL})}{|x_K - x_L|}
			 S_{KL} 
		\quad \text{ for } K \sim L.
\end{align} 
where $S_{KL}$ is a suitable average of the stationary density $\sigma$ on $K$ and $L$, i.e., 
$
	S_{KL} := \theta\big( \sigma(x_K), \sigma(x_L) \big)
$
for a fixed function $\theta : \R_+ \times \R_+ \to \R_+$ satisfying $\min\{a,b\} \leq \theta(a,b) \leq \max\{a,b\}$.
	
As \eqref{eq:discrete-FP} is the forward equation for a reversible Markov chain on $\cT$, it follows from 
the theory in \cite{maas2011gradient} and \cite{mielke2011gradient} that this equation is the gradient flow of the relative entropy $\cH_\cT:\cP(\cT) \to \R_+$ given by
\begin{align*}
	\cH_\cT(m) := \sum_{K \in \cT} m(K) \log\frac{m(K)}{\pi_\cT(K)}.
\end{align*}

The discrete analogue of \eqref{eq:defA} is given by 
the operator $\cA_{\cT}: \cP(\cT) \times \R^\cT \to \R_+$  defined by
\begin{equation}	\label{eq:defR_T}
	\cA_{\cT}(m, f) 
		= \frac14 \sum_{K, L \in \cT}
		\big(f(K) - f(L)\big)^2
		  \theta_{\log} 
		  \bigg( \frac{m(K)}{\pi_\cT(K)},
		  	     \frac{m(L)}{\pi_\cT(L)} \bigg) 
   		 w_{KL},
\end{equation}
where $\theta_{\log}(a,b) = \frac{a-b}{\log a - \log b}$ denotes the logarithmic mean.
Its Legendre transform 
$\cA_\cT^*: \cP(\cT) \times \R^\cT \to \R$ 
with respect to the second variable is given by
\begin{align*} 				
	\cA_\cT^*(m,\sigma) 
	= \sup_{f \in \R^\cT} 
		\bigg\{ 
			\sum_{K \in \cT} \sigma(K) f(K) 
			- \cA_\cT(m,f) 
		\bigg\}.
\end{align*}
In analogy to \eqref{eq:contEDI}, we can formulate the gradient flow structure for the discrete Fokker-Planck equation \eqref{eq:discrete-FP} in terms of the \emph{discrete EDI}
\begin{equation}	
\label{eq:discreteEDI}
\cH_\cT(m_T) +  \int_0^T \cA_{\cT}^*(m_t, \dot{m}_t) + \cA_{\cT}\big(m_t, -D \cH_\cT(m_t)\big) \dd t \leq \cH_\cT(m_0).
\end{equation}
The discrete counterpart of \eqref{eq:Fisher_Onsager} is the \textit{discrete Fisher information} $\cI_\cT(m)$  given by
\begin{align*}
	 \cI_\cT(m)
		 := 
	2 \cA_\cT\big( m, -D \cH_\cT (m) \big), 
	 \qquad m \in \cP(\cT).
\end{align*}

\section{Statement of the results}	\label{section:statements}

In this section we present our main result, the evolutionary $\Gamma$-convergence of the gradient flow structures in the discrete-to-continuum limit for Fokker-Planck equations on a bounded convex domain $\Omega \subset \R^d$.

Let $\cT$ be an admissible mesh on $\Omega$.
To compare measures on different spaces we introduce the canonical projection and embedding operators $\mathsf{P}_{\cT}$ and $\mathsf{Q}_{\cT}$ defined by 
\begin{equation}
	\begin{aligned}	
		\label{eq:def P projection}
		& \mathsf{P}_\cT : \cM(\bOmega) \to \cM(\cT) 
		& 
		\big(\mathsf{P}_\cT \mu\big) (K) 
		&
		 = \mu(K)
		& \text{for } K & \in \cT,\\
		& \mathsf{Q}_\cT : \cM(\cT) \to \cM(\bOmega) 
		&
		\mathsf{Q}_{\cT} m 
		& = \sum_{K \in \cT} m(K) \cU_{K}
		& \text{for } m & \in \cP(\cT).
	\end{aligned}
\end{equation}
Here, $\cU_{K}$ denotes the uniform probability measure on $K \subset \bOmega$, and $\cM(\cX)$ denotes the set of finite measures on the space $\cX$.
In particular, $\mathsf{Q}_{\cT}$ is a right-inverse of $\mathsf{P}_{\cT}$ and both mappings are mass and positivity preserving. By construction we have $\pi_\cT := \mathsf{P}_\cT {\bfm}$. 

It is also useful to introduce an operator for the piecewise constant embedding of functions:
\begin{align*}
	Q_\cT : \R^\cT \to L^\infty(\bOmega), \qquad
	\big(Q_\cT f\big) (x) = f(K) \quad  \text{for } x \in K, \ K \in \cT.
\end{align*}

Let us now consider a sequence of admissible, $\zeta$-regular meshes $\cT_N$ with mesh size $[\cT_N] \rightarrow 0$ as $N \to \infty$. To avoid towers of subscripts, we simply write $\cA_N := \cA_{\cT_N}$, $\mathsf{P}_N := \mathsf{P}_{\cT_N}$, etc. 

\subsection{Evolutionary $\Gamma$-convergence of discrete Fokker-Planck equations}

In this subsection we fix a reference probability $\bfm \in \cP(\bOmega)$ with density $\sigma(x) = \frac{\ddd \bfm}{\ddd x} = \frac{1}{Z_V} e^{-V(x)}$ as in \eqref{eq:steady-state1}.
For neighbouring cells $K, L \in \cT_N$ we fix $S_{KL} > 0$ such that 
\begin{align}
	\label{eq:S-bounds}
	\min\big\{ \sigma(x_K), \sigma(x_L) \big\} 
		\leq 
	S_{KL} 
		\leq
	\max\big\{ \sigma(x_K), \sigma(x_L) \big\} 
\end{align}
as in Section \ref{sec:gradient_flows}.

We start by collecting some conditions of the densities that will be imposed in the sequel.

\begin{definition}[Assumptions on approximating sequences]
Let $(\cT_N)_N$ be a vanishing sequence of $\zeta$-regular meshes for some $\zeta > 0$.
For a sequence of measures $m_N \in \cP(\cT_N)$ with densities $r_N = \frac{\ddd m_N}{\ddd \pi_N}$, we consider the following conditions:
\begin{enumerate}[(i)]
\item
The \emph{density lower bound} holds if, for some $\underline{k} > 0$,
\begin{align} \tag{lb}  \label{eq:lb}
		\inf_{K \in \cT_N} r_N(K) \geq \underline{k} > 0 
			\quad \forall N \in \N.
\end{align}
\item
The \emph{density upper bound} holds if, for some $\bar{k} < \infty$,
\begin{align} \tag{ub} \label{eq:ub}
		\sup_{K \in \cT_N} r_N(K)
			\leq \bar{k} < +\infty 
		\quad \forall N \in \N.
\end{align}
\item 
The \emph{neighbour continuity bound} holds if 
	\begin{align} \tag{nc} \label{eq:nc}
		\lim_{N \to \infty} 
		\sup_{
				\substack{ 
					K,L \in \cT_N \\ 
					K \sim L
				}
			}
		|r_N (K) - r_N (L)| = 0.
	\end{align}
\item 
The \emph{pointwise condition} holds if there exists a measure  $\mu \in \cP(\bOmega)$ with density $\rho = \frac{\ddd \mu}{\ddd \bfm}$ such that $\mu_N := \mathsf{Q}_N m_N \weakto \mu$ and, 
 for a.e. $x_0 \in \Omega$:
	\begin{align}	\tag{pc} \label{eq:pc}
		\lim_{\eps \rightarrow 0} 
			\liminf_{N \to \infty} 
				\sup_{x \in Q_\eps (x_0)} \rho_N(x)
	\leq \rho(x_0) 
	\leq \lim_{\eps \rightarrow 0} 
			\limsup_{N \to \infty} 
				\inf_{x \in Q_\eps (x_0)} \rho_N(x). 
	\end{align}
Here, $Q_\eps (x_0)$ denotes the open cube of side-length $\eps > 0$ centered at $x_0$, and $\rho_N(x) := r_N(K)$ for $x \in K$. 
\end{enumerate}
\end{definition}

\begin{remark}\label{rem:pointwise}
These conditions do not depend on the reference measure $\bfm$, except for the value of the constants $\underline{k}$ and $\overline{k}$.
Clearly, the pointwise condition holds if $\rho$ belongs to $C(\bOmega)$ and $\rho_n$ converges uniformly to $\rho$. 
Moreover, this condition implies subsequential pointwise convergence of $\rho_N$ to $\rho$.
\end{remark}

We now present the crucial $\Gamma$-liminf inequalities for the functionals in the EDI \eqref{eq:discreteEDI}.  

\begin{theorem}[Lower bounds for functionals] \label{thm:lower_bounds}
Let $(\cT_N)_N$  be a vanishing sequence of $\zeta$-regular meshes for some $\zeta > 0$.
The following assertions hold for any $\mu \in \cP(\bOmega)$ and $m_N \in \cP(\cT_N)$ such that $\mathsf{Q}_N m_N  \weakto \mu$ as $N \to \infty$:
\begin{enumerate}[(i)]
	\item 
	The relative entropy functionals satisfy the liminf inequality
	\begin{align} \label{eq:lower bound entropy_0}
		\liminf_{N \to \infty} \cH_N(m_N) \geq \bH(\mu).
	\end{align}
	\item 
	Assume \eqref{eq:nc}. 
	The Fisher information functionals satisfy the liminf inequality 
	\begin{align} \label{eq:lower bound Fisher info_0}
		\liminf_{N \to \infty} \cI_N(m_N) \geq \bI(\mu).
	\end{align}

	\item \label{item:speed}
	Assume \eqref{eq:lb}, \eqref{eq:ub}, and \eqref{eq:pc}. For any $\eta \in L^2(\Omega)$ and any $e_N  \in \R^{\cT_N}$ such that $\Q_N e_N \weakto \eta$ in $L^2(\Omega)$ we have
	\begin{align}	\label{eq:lower bound dual action_0}
		\liminf_{N \to \infty}  
		  \cA_N^*(m_N, e_N) \geq \bA^*(\mu, \eta).
	\end{align}
	The same bound  holds without assuming \eqref{eq:lb} if $(e_N)_N$ satisfies the additional assumption $
		\limsup_{N \to \infty} \cA_N^*(\pi_N, e_N) < +\infty.
	$
\end{enumerate}
\end{theorem}

\begin{remark} 
	We emphasize that the lower bound \eqref{eq:lb} is not required to obtain \eqref{eq:lower bound entropy_0} and \eqref{eq:lower bound Fisher info_0}.
\end{remark}
	
\begin{remark}	\label{rem:lower bound with isotropy}
	The bound \eqref{eq:lower bound dual action_0} can be obtained without assuming \eqref{eq:ub} and \eqref{eq:pc} if the mesh satisfies the so-called asymptotic isotropy condition \eqref{eq:isotropy_condition}; cf. Definition \ref{def:isotropy_local} and Proposition \ref{prop:action bounds} below. 
\end{remark}

\begin{remark}
	If $\mu \in \cP(\bOmega)$ is absolutely continuous with respect to the Lebesgue measure 
	and $m_N = \mathsf{P}_N \mu$, 
	\eqref{eq:lower bound dual action_0} can be proved under the assumptions that 
	$\eta \in \cM(\Omega)$ and 
	$e_N \in \R^{\cT_N}$ satisfies $\mathsf{Q}_N e_N \weakto \eta$ in $\cD^{'}(\Omega)$. 
	This is a consequence of an explicit construction of a recovery sequence for the action $\cA_N(m_N, \cdot)$ (as in the isotropic case in Proposition \ref{prop:action bounds}); cf. Remark \ref{rem:limsup}.
\end{remark}

Using Theorem \ref{thm:lower_bounds} we obtain our main result, the evolutionary $\Gamma$-convergence of the discrete gradient flow structures. 
The following result shows that one can pass to the limit in each of the terms of the discrete gradient flow formulation \eqref{eq:discreteEDI} and recover the Wasserstein gradient flow structure as a consequence.

\begin{theorem}[Evolutionary $\Gamma$-convergence]
\label{thm:MAIN}
Let $T > 0$ and consider a vanishing sequence of $\zeta$-admissible meshes $(\cT_N)_N$.  
Fix an initial measure $\mu_0 \in \cP(\bOmega)$ such that $\bH(\mu_0) < +\infty$, together with measures $m_0^N \in \cP(\cT_N)$ for $N \geq 1$, that are well-prepared in the sense that 
\begin{align*}
	\mathsf{Q}_N m_0^N \weakto \mu_0
	\tand
	\lim_{N \to \infty} \cH_N(m_0^N) = \bH(\mu_0).
\end{align*}
For each $N \geq 1$, let $(m_t^N)_{t \in [0,T]}$ be the solution to the discrete Fokker-Planck equation \eqref{eq:discrete-FP} with initial datum $m_0^N$, which satisfies the EDI
\begin{align*}
	\cH_N(m_t^N) 
			+  \int_0^T \cA_N^*(m_t^N, \dot{m}_t^N) 
					  + \cA_N\big(m_t^N, -D \cH_N(m_t^N)\big) \dd t 
	\leq \cH_N(m_0^N).
\end{align*}
Then: 
\begin{enumerate}[(i)]
	\item The sequence of curves $(\mu^N )_N$ defined by $\mu_t^N:= \mathsf{Q}_N m_t^N$ is compact in the space $C\big([0,T];(\cP(\bOmega),\bW) \big)$. 
	Thus, up to a subsequence, we have
	\begin{equation}	
		\label{eq:thm_compactness}
			\sup_{t \in [0,T]}
			\bW\big( \mu_t^N , \mu_t \big) \to 0
				\quad \text{as }N \to \infty.
	\end{equation}
	\item The following estimates hold:
	\begin{subequations}	\label{eq:all}
	\begin{align} 
		\label{item:energy}
		& \textrm{Entropy:}   &
			\liminf_{N \to \infty} \cH_N(m_t^N) 
			 &
			\geq  \bH(\mu_t) 
				\quad \forall t \in [0,T], \\
		\label{item:Fisher} 
		& \textrm{Fisher I.:}   &
		\liminf_{N \to \infty} 
			\int_0^T \cA_N\big(m_t^N, -D \cH_N(m_t^N)\big) \dd t 
			 &
			\geq  \int_0^T\bA\big(\mu_t, -D \bH(\mu_t)\big) \dd t,\\
		\label{item:metric derivative}  				
		& \textrm{Speed:}   &
			\liminf_{N \to \infty}  
			\int_0^T \cA_N^*(m_t^N, \dot{m}_t^N) \dd t
			 &
			\geq  \int_0^T \bA^*(\mu_t, \dot{\mu}_t)   \dd t . 
	\end{align}
	\end{subequations}
	\item The curve $(\mu_t)$ solves the EDI \eqref{eq:contEDI}, and hence, the Fokker-Planck equation \eqref{eq:intro_FP}.
\end{enumerate}
\end{theorem}

\begin{remark}\label{rem:well-prepared}
	The well-preparedness assumption holds in the special case where the discrete measures are defined by $m_0^N := \mathsf{P}_N \mu_0$ as in \eqref{eq:def P projection}. 
	Indeed, in that case we have $\cH_N(m_0^N) = \bH(\mathsf{Q}_N \mathsf{P}_N \mu_0)$, so that the convergence of the relative entropy functionals follows from Jensen's inequality.
\end{remark}

The proofs of Theorem \ref{thm:lower_bounds} and Theorem \ref{thm:MAIN} appear in Section \ref{section:evolut_gamma_Wass}.
They rely on a Mosco convergence result for discrete energy functionals of independent interest, which we will now describe.

\subsection{Mosco convergence of Dirichlet energy functionals}

Fix an absolutely continuous probability measure $\mu \in \cP(\bOmega)$, 
and assume that its density $\upsilon$ with respect to the  Lebesgue measure on $\bOmega$ satisfies the two-sided bounds
\begin{align*}
	0 
	< \underline{c} 
	\leq \upsilon(x)
	\leq \overline{c} 
	< \infty
	\text{ for all } x \in \bOmega.	
\end{align*}
We consider the \emph{continuous Dirichlet energy} $\bF_\mu : L^2(\Omega) \rightarrow \R_+ \cup \{ +\infty \} $ given by
\begin{equation}
\bF_\mu(\vphi) := 
\bA(\mu, \vphi)
= \begin{cases} 
\displaystyle	\frac{1}{2} \int_{\Omega} | \nabla \vphi|^2 \dd \mu &\text{if }\vphi \in H^1(\Omega),\\
+\infty &\text{otherwise}
\end{cases}
\end{equation}
where $\bA$ is defined in \eqref{eq:defA}. 

Similarly, for a $\zeta$-regular mesh $\cT$ and a probability measure $m \in \cP(\cT)$, we consider the \emph{discrete Dirichlet energy} $\cF_\cT : \R^\cT \to \R_+$ defined by
\begin{equation}
	\cF_\cT(f) = 
		\frac{1}{4} 
			\sum_{K,L \in \cT} 
				\big( f(K) - f(L) \big)^2
				U_{KL}
				\frac{\sH^{d-1}(\Gamma_{KL})}{|x_K - x_L|}	
\end{equation}
where 
$
\min\big\{ 	\frac{m(K)}{|K|}, 
			\frac{m(L)}{|L|} 
	\big\} 
\leq 
U_{KL} 
\leq
\max\big\{ 	\frac{m(K)}{|K|}, 
			\frac{m(L)}{|L|} 
	\big\}
$.
In the special case where 
$U_{KL}$ is defined in terms of the logarithmic mean of $r_K$ and $r_L$, namely, 
$U_{KL} = \frac{r_K - r_L}{\log r_K - \log r_L}S_{KL}$ with $r_K = \frac{m(K)}{\pi_\cT(K)}$, 
this functional is related to the functional $\cA_\cT$ by
\begin{align}
	\label{eq:defFT}
			\cF_\cT(f) := \cA_\cT(m, f).
\end{align}
To compare the discrete and the continuous functionals we consider the embedded funtionals $\tbF_\cT : L^2(\Omega) \to \R_+ \cup \{ + \infty \} $ defined by
\begin{equation}		\label{eq:F_N-tilde}
\tbF_\cT(\vphi) := \begin{cases}
\cF_\cT \big( P_\cT \vphi \big) &\text{if }\vphi \in \PCcT, \\
+\infty &\text{otherwise},
\end{cases}
\end{equation}
where $\PCcT$ denotes the space of all functions in $L^2(\Omega)$ that are constant a.e.~on each cell $K \in \cT$, and 
\begin{align}
	\label{eq:PT}
	\big(P_\cT \vphi\big)(K) := \vphi(x_K) 
	\quad \text{for }
	\vphi : \Omega \to \R.
\end{align}

We then obtain the following convergence result. For the definition of Mosco convergence we refer to Definition \ref{def:Gamma} below.

\begin{theorem}[Mosco convergence] \label{thm:Mosco}
	Let $(\cT_N)_N$ be a vanishing sequence of $\zeta$-regular meshes, 
	and suppose that $\mu$ and $(m_N)_N$ satisfy \eqref{eq:lb}, \eqref{eq:ub}, and \eqref{eq:pc}. 
	Then we have Mosco convergence 
	$\tbF_{\cT_N}  \xrightarrow[]{M} \bF_\mu$ with respect to the $L^2(\Omega)$-topology.
\end{theorem}

The proof of this result follows the strategy developed in \cite{Alicandro2004}, where similar $\Gamma$-convergence results have been obtained for more general energy functionals on a particular mesh (the cartesian grid).
In that paper, the authors do not explicitly characterise the limiting functional, 
except in special situations, such as the periodic setting. 
For our application to evolutionary $\Gamma$-convergence, a characterisation of the limiting functional is crucial.

\begin{remark}\label{rem:semigroup}
	Mosco convergence of Dirichlet energy functionals is equivalent to strong convergence of the associated semigroups \cite{Mosco:1994}; see also \cite{kuwae2003} for a generalisation to Dirichlet forms defined on different spaces.
	\end{remark}
	
\subsection{Related work}	\label{section:previous_works}

We close this section with some comments on related work.

\subsubsection*{Convergence of the discrete Fokker-Planck equations}
It is well known that the discrete heat flow converges to the continuous heat flow for any sequence of admissible meshes with vanishing diameter. 
The authors in \cite{eymard2000finite}, \cite{barth2018finite} exploit classical Sobolev a priori estimates and pass to the limit in the weak formulation of the equation, in dimension 2 and 3 (see \cite[Lemma 8]{barth2018finite}). 
A unified framework for discretisation of partial differential equations in higher dimension can be found in \cite{droniou2018gradient}.
Convergence results for finite-volume discretisations of Fokker-Planck equations based on different Stolarsky means have recently been obtained in \cite{Heida-Kantner-Stephan:2020}.

\subsubsection*{Entropy gradient flows in discrete settings}	\label{section:discrete_OT}

Entropy gradient flow structures for discrete dynamics have been intensively studied in discrete settings following the papers \cite{chow12,maas2011gradient,mielke2011gradient}.
Many subsequent works deal with connections to curvature bounds and functional inequalities \cite{erbar12,mielke2013geodesic,erbar14,erbar15,fathi2016entropic,erbar19}. 
Entropy gradient flow structures have also been exploited to analyse the discrete-to-continuum limit from several perspectives, see, e.g., \cite{Cances-Guichard:2017,Cances-Gallouet-Laborde:2019,Chalub-Monsaingeon-Ribeiro:2019,Bruna-Burger-Carrillo:2020}.

\subsubsection*{Evolutionary $\Gamma$-convergence for Fokker-Planck in 1D}

Evolutionary $\Gamma$-convergence of the discrete gradient flow structures for Fokker-Planck equations has been proved in the one-dimensional setting under additional geometric conditions using methods that do not extend straightforwardly to higher dimensions \cite{disser2015}.

In particular, the authors work with meshes that satisfy the \textit{center of mass condition}
\begin{align} 	\label{eq:centerofmass}
	\avint_{\Gamma_{KL}} x \dd \Haus^{d-1} = \frac{x_K+x_L}{2}, \quad \text{for all } K \sim L \in \cT.
\end{align}
This condition implies the Gromov-Hausdorff convergence of the associated transport metrics \cite{gladbach2018}.  
Here, we work with more general meshes for which Gromov-Hausdorff convergence of the associated transport metrics does not always hold.

Moreover, in one dimension, it is possible to construct explicit solutions to the continuity equation from discrete vector fields using linear interpolation techniques.
As such methods are not available in higher dimensions, we take a more variational approach in this paper.

\subsubsection*{Scaling limits for discrete optimal transport in any dimension.}

The crucial liminf inequality \eqref{eq:lower bound dual action_0} can be proved under weaker assumptions on the approximating sequence of measures if the meshes satisfy a suitable isotropy condition, which we will now recall.

\begin{definition}[Asymptotic isotropy]
	\label{def:isotropy_local} 
	A vanishing sequence of meshes $(\cT_N)_N$ is said to satisfy the \emph{asymptotic isotropy condition} if, for every $N \in \N$,
	\begin{equation}	\label{eq:isotropy_condition}
		\frac{1}{2} \sum_{L \in \cT_N}  w_{KL} \left( x_K - x_L \right) \otimes \left( x_K - x_L \right) \leq \pi_N(K) \big( I_d + \eta_{\cT_N}(K) \big) 		\qquad \forall K \in \cT_N,
	\end{equation}
	where $\sup\limits_{K \in \cT_N} \| \eta_{\cT}(K) \| \rightarrow 0$ as $N \to \infty$.
\end{definition}

Under this condition, the following following version of \eqref{eq:lower bound dual action_0} has been proved in \cite[Proposition 6.6]{gladbach2018}. In that paper the reference measure is the Lebesgue measure.
Here we formulate a slight generalisation with the reference measure $\bfm$.

\begin{proposition}[Action bounds]	\label{prop:action bounds}
	Let $(\cT_N)_N$ be a vanishing sequence of meshes satisfying the asymptotic isotropy condition \eqref{eq:isotropy_condition}. 
	Let $\mu \in \cP(\bOmega)$ 
	and $\eta \in \cM_0(\bOmega)$,  
	and suppose that 
	$m_N \in \cP(\cT_N)$ and 
	$e_N \in \cM_0(\cT_N)$ satisfy
	$\mathsf{Q}_N m_N \weakto \mu$ and 
	$\mathsf{Q}_N e_N \weakto \eta$
	as $N \to \infty$. 
	Then we have the lower bound
	\begin{equation}	
	\label{eq:lower bound dual action}
		\liminf_{N \to \infty} 
					\cA^*_N(m_N, e_N) \geq \bA^*(\mu,\eta).
	\end{equation}
	\end{proposition}

It has also been shown in \cite{gladbach2018} that Gromov--Hausdorff convergence of the associated transport distances holds under the asymptotic isotropy condition; 
see also \cite{gladbach2019} for a study of the limiting metric in the one-dimensional periodic setting.
In the current paper we do not assume that the discrete meshes satisfy an isotropy condition.

\subsection*{Notation}
Throughout the paper we use the notation $a \lesssim b$ (or $b \gtrsim a$) if $a \leq C b$ with $C < \infty$ depending only on $\Omega$, $\zeta$, and $\bfm$. We write $a \eqsim b$ if $a \lesssim b$ and $a \gtrsim b$.

\section{Proof of the main result: the Wasserstein evolutionary \texorpdfstring{$\Gamma$}{Gamma}-convergence}
\label{section:evolut_gamma_Wass}

In this section we prove our main result, the evolutionary $\Gamma$-convergence of the discrete gradient flow structures (Theorem \ref{thm:MAIN}). The section is divided into three parts: 
the first subsection concerns the proof of Theorem \ref{thm:lower_bounds}, which relies on Theorem \ref{thm:Mosco}. 
The second subsection contains a proof of compactness for the continuously embedded discrete solutions.
In the third and final part we complete the proof of Theorem \ref{thm:MAIN}.

\subsection{Asymptotic lower bounds for the functionals}

Let $\mu$ and $m_N$ be as in the statement of Theorem \ref{thm:lower_bounds}. 
Write $\mu_N := \mathsf{Q}_N m_N$ and let $\rho_N$ be the density of $\mu_N$ with respect to $\bfm$. 

\begin{proof}[Proof of Theorem \ref{thm:lower_bounds}]

The proof consists of three parts.

\smallskip

\emph{ ({i}) Lower bound for the entropy.}
Note that  $\cH_N(m_N) = \Ent(\mu_N | \mathsf{Q}_N \pi_N)$ and $\bH(\mu) = \Ent( \mu | \bfm )$, where $\Ent(\cdot|\cdot)$ denotes the relative entropy.
Since $\mu_N \weakto \mu$ and $\mathsf{Q}_N \pi_N \weakto \bfm$, the result follows immediately from the joint weak lower semicontinuity of $\Ent(\cdot|\cdot)$, see, e.g., \cite[Lemma 9.4.3]{AmbrosioGigliSavare08}.

\smallskip

\emph{ ({ii}) Lower bound for the Fisher information.}
Assume that \eqref{eq:nc} holds. 
We first prove the lower bound \eqref{eq:lower bound Fisher info_0} under the additional assumption \eqref{eq:lb}. This assumption will be removed at the end of the proof. 
The key identity for the Fisher information is
\begin{align} \label{eq:tildeAformrepres}
	 \tilde \cA_N \big( m_N, - D \cH_N(m_N) \big) 
	 	= 4 \cE_N \big( \sqrt{r_N} \big) ,
\end{align}
where $\cE_N(f):=\cA_N(\pi_N,f)$ is the discrete Dirichlet energy with reference measure $\pi_N$, and $\tilde \cA_N$ is defined by replacing the logarithmic mean $\theta_{\log}$ in the definition of $\cA_N$ by $\tilde \theta(a,b) := \theta_{\log}(\sqrt a, \sqrt b)^2$.
Since $\min\{a, b\} \leq \tilde \theta(a,b), \theta_{\log}(a,b) \leq \max\{a, b\}$, we have
\begin{align*}
	|  \theta_{\log}(a,b) - \tilde \theta(a,b) |
		\leq | a-b |
		\leq \frac{| a-b | }{\min\{a, b\}} \tilde \theta(a,b).
\end{align*}
The assumptions \eqref{eq:nc} and \eqref{eq:lb} yield
\begin{align}	\label{eq:cont_proof}
	\eps_N := 
			\sup_{
				\substack{ 
					K,L \in \cT_N \\ 
					K \sim L
				}
			}	
			\abs{r_N(K) - r_N(L)} \to 0
		\quad \text{and} \quad
				\inf_{K \in \cT_N } r_N(K) \geq \underline k,
\end{align}
Using these estimates and the identity $(\log a - \log b)^2 \tilde \theta(a,b) = 4 \big(\sqrt{a} - \sqrt{b}\big)^2$ we obtain
\begin{equation}
\begin{aligned}
	\label{eq:Fisher-Dirichlet}
	 \big| \tfrac12\cI_N(m_N) - 4\cE_N(\sqrt{r_N} ) \big|
		& = \big|  
				\big( \cA_N - \tilde \cA_N \big) 
				\big( m_N, - D \cH_N(m_N) \big) 
			\big|
	\\ & = \frac14 \sum_{K, L \in \cT_N} w_{KL}		
	\big(\log r_N(K) - \log r_N(L)\big)^2 
	\\& \qquad \times \Big(  
			\theta_{\log} \big( r_N(K), r_N(L) \big) - 
			\tilde \theta_{\log} \big( r_N(K), r_N(L) \big) 
		\Big)
	\\ & 
	\leq   \frac{4 \eps_N}{\underline k} \cE_N\big(\sqrt{r_N}\big).
\end{aligned}
\end{equation}

Let us now assume that $\sup_{N} \cI_N(m_N) < \infty$ along a subsequence; if this were not the case, the result holds trivially. 
The previous bound implies that also $\sup_{N} \cE_N\big(\sqrt{r_N}\big) < \infty$, hence $\big(\sqrt{\rho_N}\big)_N$ has a subsequence that converges strongly in $L^2(\Omega, \bfm)$ by Proposition \ref{prop:sobolev_limit} below. Let $g \in L^2(\Omega, \bfm)$ be its limit.
Since $\| \rho_N - g^2 \|_{L^1} \leq \| \sqrt{\rho_N} - g\|_{L^2} \|\sqrt{\rho_N} + g \|_{L^2}$, we infer that $\rho_N \to g^2$ in $L^1(\Omega, \bfm)$. 
As $\mu_N = \rho_N \bfm \weakto \mu$ in $\cP(\bOmega)$ by assumption, we infer that $\mu = \rho \bfm$ with $\rho := g^2$.
Now we apply \eqref{eq:Fisher-Dirichlet} and the Mosco convergence from Theorem \ref{thm:Mosco} to obtain
\begin{align*}
	\liminf_{N \to \infty}  
				\cI_N(m_N) 
		\geq  \liminf_{N \to \infty} 
				8 \cE_N \big( \sqrt{r_N} \big) 
		\geq    8 \bA \big( \bfm, \sqrt \rho\big) 
		   =  \bI(\mu),
\end{align*}
which concludes the proof. 

Let us now show how to remove the assumption \eqref{eq:lb}.
The argument is based on the convexity of $m \mapsto \cI_N(m)$, which is a consequence of the joint convexity of the map $(a,b) \mapsto (a - b) (\log a - \log b)$ on $(0,\infty) \times (0,\infty)$. 

Pick $\delta \in (0,1)$ and set $m_N^\delta := (1-\delta) m_N + \delta \pi_N$.
Note that $m_N^\delta$ satisfies \eqref{eq:lb} with $\underline k = \delta$. Moreover, $\mathsf{Q}_N m_N^\delta \weakto \mu^\delta := (1-\delta) \mu + \delta \bfm$. 
Applying the first part of the result we obtain
\begin{align*}
	\bI(\mu^\delta) 
		\leq   \liminf_{N \to \infty} \cI_N(m_N^\delta)	
		\leq	(1-\delta) \liminf_{N \to \infty} \cI_N(m_N)
\end{align*}
for every $\delta \in (0,1]$, where the last inequality uses the convexity of $\cI_N$ and the fact that $\cI_N(\pi_N) = 0$.
Since $\mu^\delta \weakto \mu$, the result follows from the lower semicontinuity of $\bI$ with respect to the weak convergence in $\cP(\bOmega)$; see \cite[Lemma 2.2]{Gianazza-Savare-Toscani:2009}.

\smallskip
\emph{ ({iii}) Lower bound for $\cA_N^*$.}
Assume first that \eqref{eq:lb}, \eqref{eq:ub}, and \eqref{eq:pc} hold. 
Fix $\eta \in L^2(\Omega, \bfm)$ and let $e_N \in  L^2(\cT_N, \pi_N)$ be such that $\Q_N e_N \weakto \eta$ in $L^2(\Omega, \bfm)$.
Theorem \ref{thm:Mosco} (in particular, the existence of a recovery sequence) implies that for every $\vphi \in C_c(\Omega)$  there exist $f_N \in L^2(\cT_N, \pi_N)$ such that $\Q_N f_N \rightarrow \vphi$ in $L^2(\Omega, \bfm)$ and
\begin{align*}
	\limsup_{N \to \infty} 
		 \cA_N(m_N, f_N) 
	\leq \bA(\mu, \vphi).
\end{align*}
Since $\Q_N e_N \weakto \eta$ in $L^2(\Omega, \bfm)$, it follows that $\ip{e_N, f_N}_{L^2(\cT_N, \pi_N)} \to \ip{\eta, \vphi}_{L^2(\Omega, \bfm)}$ and
\begin{align*}
	\ip{\eta, \vphi}_{L^2(\Omega, \bfm)} - \bA(\mu, \vphi)
	& \leq \liminf_{N \to \infty}
		\ip{e_N, f_N}_{L^2(\cT_N, \pi_N)} - \cA_N(m_N, f_N)
\\	& \leq \liminf_{N \to \infty}
		\cA_N^*(m_N, e_N).
\end{align*}
Taking the supremum over $\vphi$, we infer that $\bA^*(\mu, \eta) \leq \liminf_{N \to \infty} \cA_N^*(m_N, e_N)$, as desired.

Assume now that \eqref{eq:ub}, \eqref{eq:pc} hold, and that $E := \limsup_{N \to \infty} \cA_N^*(\pi_N, e_N) < + \infty$, instead of \eqref{eq:lb}.
The key observation is that the map $m_N \mapsto \cA_N^*(m_N, e_N)$ is convex: indeed, the concavity of $\theta_{\log}$ implies the concavity of $m_N \mapsto \cA(m_N, f_N)$, and thus the convexity of its Legendre dual as a supremum of convex maps. 
To take advantage of this fact, we fix $\delta \in (0,1)$ and define $m_N^\delta := (1-\delta) m_N + \delta \pi_N$.
Note that $\mathsf{Q}_N m_N^\delta \weakto  \mu^\delta := (1-\delta) \mu + \delta \bfm$ and $m_N^\delta$ satisfies \eqref{eq:lb} with $\underline{k} = \delta$. 
We may thus apply the first part of the result and the convexity to obtain
\begin{align*}
	\bA^*(\mu^\delta, \eta) 
	& \leq \liminf_{N \to \infty} 
		\cA_N^*(m_N^{\delta}, e_N)
	\\ & 	
	\leq \liminf_{N \to \infty} 
	(1 - \delta) \cA_N^*(m_N, e_N)  
	+ \delta \cA_N^*(\pi_N, e_N)		
\\	& \leq		(1 - \delta)\Big( \liminf_{N \to \infty}  \cA_N^*(m_N, e_N)  \Big)
	+ \delta E.
\end{align*} 
Using the weak lower semicontinuity of $\bA^*(\cdot, \eta)$, we obtain the desired inequality \eqref{eq:lower bound dual action_0} by passing to the limit $\delta \rightarrow 0$.
\end{proof}

\subsection{Compactness and space-time regularity}
In this section we prove the compactness of the family of embedded discrete gradient flow curves $(t \mapsto \mu_t^N)_N$ in the space
$
C\big( [0, T]; (\cP(\bOmega), \bW)\big).
$
We follow the strategy of \cite[Theorem~3.1]{liero2015microscopic}, which is based on a metric Ascoli-Arzel\'a theorem. 
The corresponding one-dimensional result has been obtained in \cite{disser2015} using explicit interpolation formulas that are not available in the multi-dimensional setting.

Our proof is based on the following coarse energy bound from \cite[Lemma 3.4]{gladbach2018}. 
Here and below, $(H_t)_{t \geq 0}$ denotes the Neumann heat semigroup on $\Omega$. 
Moreover, $ \cM_0(\cT)$ denotes the space of signed measure on $\cT$ with zero total mass.

\begin{lemma}[Coarse energy bound]		
\label{lemma:coarse bound dual action}
Fix $\zeta \in (0, 1]$. 
There exists a constant $C < \infty$ such that for any  $\zeta$-regular mesh $\cT$, for any $m \in \cP(\cT)$ and any $\sigma \in \cM_0(\cT)$,  we have
\begin{equation}	
\label{eq:coarse bound dual action}
	\bA^*
		\big( H_{[\cT]} \mathsf{Q}_{\cT} m, 
			  H_{[\cT]} \mathsf{Q}_{\cT} \sigma
		\big) 
			\leq C \cA^*_{\cT}(m, \sigma) .
\end{equation}
\end{lemma}

Let us stress that this result holds without any isotropy assumption on the mesh.

\begin{lemma}[$\bW$-Equi-continuity]	\label{lemma:equi-continuity in W2}
Let $\{ \cT_N\}_N$ be a vanishing sequence of $\zeta$-regular meshes. 
For each $N \in \N$, let $(m_t^N)_{t \in [0,T]}$ be a continuous curve in $\cP(\cT_N)$, and 
suppose that the following uniform energy bound holds:
\begin{equation} 	
	\label{eq:uniform discrete action bound hp}
	A := \sup_N 
			\int_0^T
				\cA^*_N\big( m_t^N, \dot{m}_t^N \big)
			\dd t
		 < +\infty.
\end{equation}
Then the curves $\tilde\mu^N : [0,T] \to \big(\cP(\bOmega), \bW\big)$
defined by $\tilde\mu_t^N :=  H_{[\cT_N]}  \mathsf{Q}_N m_t^N$ are equi-$\frac{1}{2}$-H\"older continuous, i.e., for $0 \leq s < t \leq T$ we have
\begin{equation}	\label{eq:equi-holder-W1}
	\bW\big( \tilde\mu_t^N, \tilde\mu_s^N\big) 
			\lesssim \sqrt{A(t-s)}.
\end{equation} 
\end{lemma}

\begin{proof}
For $0 \leq s \leq t \leq T$ we invoke the Benamou-Brenier formula \eqref{eq:def W2 dynamical} and Lemma \ref{lemma:coarse bound dual action} to obtain
\begin{align*}
\bW^2 \big( \tilde{\mu}_t^N, \tilde{\mu}_s^N \big) 
	& \leq (t-s) 
		 \int_s^t \bA^*
		 	\big(
				\tilde{\mu}_h^N, \partial_h \tilde{\mu}_h^N
			\big) 
		 \dd h
\\&	\lesssim (t-s) 
		\sup_N
			\int_0^T 
				 \cA^*_N \big( m_h^N, \partial_h{m}_h^N \big) 
			\dd h 
		\leq A (t-s),
\end{align*}
which concludes the proof. 
\end{proof}

A corollary of Lemma \ref{lemma:equi-continuity in W2} is the following compactness and regularity result.

\begin{proposition}[Compactness and regularity]\label{prop:compactness1}
For $t \in [0,T]$ and $N \geq 1$, 
let $\mu_t^N := \mathsf{Q}_N m_t^N \in \cP(\bOmega)$ be defined as in Theorem \ref{thm:MAIN}, 
and let $\rho_t^N$ be the density of $\mu_t^N$ 
with respect to $\bfm$.
There exists a $\bW$-continuous curve $t \mapsto \mu_t \in \cP(\bOmega)$ satisfying, up to a subsequence,
\[
	\sup_{t \in [0,T]}
		\bW \big(\mu_t^N, \mu_t \big) \to 0 
			\quad \text{as }N \rightarrow +\infty.
\]
\end{proposition}

\begin{proof}
We apply Lemma \ref{lemma:equi-continuity in W2} to the family of discrete gradient flow solutions $(t \mapsto m_t^N)_N$. 
In this case, the required estimate \eqref{eq:uniform discrete action bound hp} follows directly from the discrete EDI \eqref{eq:discreteEDI} 
and the well-preparedness of the initial conditions $(m_0^N)_N$.
Thus, Lemma \ref{lemma:equi-continuity in W2} implies the $\bW$-equi-continuity of the curves $(\tilde \mu^N)_N$ defined by $\tilde\mu_t^N :=  H_{\eps_N} \mathsf{Q}_N m_t^N$, where $\eps_N := [\cT_N]$.
The metric Arzel\'a-Ascoli Theorem \cite[Proposition~3.3.1]{AmbrosioGigliSavare08} yields the existence of a limiting curve $t \mapsto \mu_t$ satisfying $\sup_t \bW(\tilde\mu_t^N, \mu_t) \to 0$ as $N \to\infty$. 
Using the well-known heat flow bound
$
	\bW(\tilde\mu_t^N , \mu_t^N) \leq C \sqrt{\eps_N}
$
 (see, e.g., \cite[Lemma 2.2(iii)]{gladbach2018} for a proof), we obtain the desired result.
\end{proof}

\subsection{Proof of the Wasserstein evolutionary $\Gamma$-convergence}

We are finally ready to give the proof of the main result.
In order to apply the liminf inequalities from Theorem \ref{thm:lower_bounds} we use regularity properties of the discrete Fokker-Planck equation that can be derived from much more general results; see \cite{chen2019heat2} for Harnack inequalities and \cite[Theorem 1.20]{chen2016stability} for ultracontractivity. 

\begin{proposition}[Regularity of the discrete flows] 	\label{prop:hold_discrete}
Let $\cT$ be a $\zeta$-regular mesh, 
let $(m_t)_{t}$ $\subset \cP(\cT)$ be a solution to the discrete Fokker-Planck equation, and set $r_t := \frac{\ddd m_t}{\ddd \pi}$. 
\begin{enumerate}[(i)]
\item For any $t > 0$ there exist $C = C(\Omega, \bfm, \zeta, t) < \infty$ and $\lambda = \lambda(\Omega, \bfm, \zeta) > 0$ such that the following H\"older estimate holds:
\begin{align}
	\label{eq:Holder}
	\abs{r_t(K) - r_t(L)}
		 \leq C |x_K - x_L|^{\lambda} 
		 	 \sup_{K' \in \cT} | r_{t/2}(K')| 
		 \quad \forall K,L \in \cT.
\end{align}
\item For any $t > 0$ the ultracontractivity estimate 
\begin{align}	
	\label{eq:ultracontractivity_discrete}
	\| r_t \|_{L^\infty(\pi_\cT)} 
		\leq C 
			\big( 1 \vee t^{-\frac{d}{2}} \big)
				 \| r_0 \|_{L^1(\pi_\cT)}
\end{align}
holds with $C = C(\Omega, \bfm,  \zeta) < \infty$. 
\end{enumerate}
\end{proposition}

We stress that the constants depend only on the aforementioned parameters.

We will also use the following auxiliary result from \cite[Corollary 4.4]{stefanelli2008brezis}. 

\begin{proposition}[Evolutionary $\Gamma$-liminf inequality]	\label{prop:evolGammaliminf}
	Let $\cX$ be a separable Hilbert space and fix $T > 0$. Let $g_N, g_\infty : (0,T) \times \cX \rightarrow [0,+\infty]$ be such that, for a.e. $t \in (0,T)$, 
	\begin{enumerate}[(i)]
		\item $g_N(t, \cdot), g_\infty(t, \cdot)$ are convex and lower semicontinuous;
		\item $\displaystyle g_\infty(t,\vphi) 
			\leq \inf 
				\big\{ 
					\liminf_{N \to \infty} 
					g_N(t,\vphi_N) 
						\suchthat 
					\vphi_N \weakto \vphi 
					\text{ in } \cX 
				\big\}
		$ for all $\vphi \in \cX$.
	\end{enumerate}
Then, for any $\vphi_N, \vphi \in L^2(0,T; \cX)$ with $\vphi_N \weakto \vphi$ in $L^2(0,T; X)$, we have
\begin{align}	\label{eq:liminf_condition_evol}
	\int_0^T g_\infty\big(t,\vphi(t)\big) \dd t 
		\leq \liminf_{N \to \infty} 
			\int_0^T g_N\big(t, \vphi_N(t)\big) \dd t.
\end{align}
\end{proposition}

\medskip

\begin{proof}[Proof of Theorem \ref{thm:MAIN}]

	\emph{(i)}:
	The compactness of $(\mu^N)$ in $C\big([0,T];(\cP(\bOmega),\bW) \big)$ follows from Proposition \ref{prop:compactness1}.

	\smallskip

	\emph{(ii)}:
	We prove the inequalities in \eqref{eq:all}. 
	The inequalities \eqref{item:energy} and \eqref{item:Fisher} follow straightforwardly from the bounds of Theorem \ref{thm:lower_bounds}.
	More work is required to prove \eqref{item:metric derivative}, as we only have time-averaged bounds on $\cA_N^*(m_t^N, \dot m_t^N)$ along the discrete flows. 
	Here we proceed using Proposition \ref{prop:evolGammaliminf}.

	\smallskip
	\emph{Evolutionary lower bound for the relative entropy \eqref{item:energy}.} 
	In view of the weak convergence $\mathsf{Q}_N m_t^N \weakto \mu_t$, this bound follows from the liminf inequality for the entropies \eqref{eq:lower bound entropy_0} obtained in Theorem \ref{thm:lower_bounds}. 
	
	\smallskip
	\emph{Evolutionary lower bound for the Fisher information \eqref{item:Fisher}.}
	It follows from the H\"older regularity result in Proposition \ref{prop:hold_discrete} that the sequence of discrete measures $(m_t^N)_N$ satisfies \eqref{eq:nc} for any $t \in (0,T]$. 
	Consequently, $\liminf_{N \to \infty} \cI_N(m_t^N) \geq \bI(\mu_t)$ by the liminf inequality for the relative Fisher information \eqref{eq:lower bound Fisher info_0} obtained in Theorem \ref{thm:lower_bounds}. 
	Therefore, the desired inequality \eqref{item:Fisher} follows from Fatou's Lemma.

	\smallskip

	\emph{Evolutionary lower bound for the metric derivative \eqref{item:metric derivative}.} \label{subsubsec:metricderivative_evol_lb}
	To ensure that our densities are bounded away from $0$, 
	we set 
	\begin{align*}
		m_t^{N,\alpha} := (1-\alpha) m_t^N + \alpha \pi_N
			\tand
		\mu_t^{\alpha} := (1-\alpha) \mu_t + \alpha \bfm
	\end{align*}
	for $\alpha \in (0,1)$.
	Fix $0 < \delta < (1 \wedge T)$ and define $g_N, g_\infty : (\delta, T) \times L^2(\Omega, \bfm) \to [0, + \infty]$ by 
	\begin{align*}
		g_N^\alpha(t,\vphi):= 
			\begin{cases}
				\cA_N^*
				\big(
					m_t^{N,\alpha}, 
					(P_N \vphi)\pi_N
				\big) 
				&\text{if }\vphi \in \PCN \\
				+\infty &\text{otherwise}
			\end{cases},
			\qquad
			g_{\infty}^\alpha(t, \vphi) 
				:= 
			\bA^*(\mu_t^\alpha, \vphi \bfm).
	\end{align*} 
	We will check that the conditions (i) and (ii) of Proposition \ref{prop:evolGammaliminf} are satisfied. 

	\smallskip

	\noindent\textbf{Step 1.} \emph{Verification of conditions (i) and (ii)}.
	
	\smallskip

	Clearly, the maps $g_N^\alpha(t, \cdot)$ 
	are convex and lower semicontinuous in $L^2(\Omega, \bfm)$ 
	for every $t \in (\delta,T)$, 
	which shows that condition (i) holds.
	
	To verify condition (ii), 
	we pick $\eta \in L^2(\Omega)$ and $e_N \in \R^{\cT_N}$ 
	such that $Q_N e_N \weakto \eta$ in $L^2(\Omega)$. 
	We have to show that 
	$\liminf_{N \to \infty} 
		\cA_N^*(m_t^{N,\alpha}, e_N) 
	\geq \bA^*(\mu_t^\alpha, \eta)$.
	To show this, we will check the conditions \eqref{eq:ub}, \eqref{eq:lb}, and \eqref{eq:pc} of Theorem \ref{thm:lower_bounds}\eqref{item:speed}. 

	\smallskip

	\noindent \textbf{Step 2.} \emph{Verification of \eqref{eq:ub}, \eqref{eq:lb}, and \eqref{eq:pc}.}
	
	\smallskip

	By construction, $(m_t^{N,\alpha})_N$ clearly satisfies \eqref{eq:lb}.
	Moreover, the hypercontractivity estimate from Proposition \ref{prop:hold_discrete} implies that $(m_t^{N,\alpha})_N$ satisfies \eqref{eq:ub}.
	Therefore, it remains to show that $(m_t^{N,\alpha})_N$ satisfies \eqref{eq:pc}. 
	Clearly, it suffices to prove that this property holds for $(m_t^N)$.
	
	To show this, we fix $x_0 \in \Omega$ and $\eps > 0$.
	Let $\rho_t^N$ be the density of $\mu_t^N$ with respect to $\bfm$.
	Using the H\"older regularity and the hypercontractivity result from Proposition \ref{prop:hold_discrete}, we infer that
	\begin{align*}
		|\rho_t^N(x) - \rho_t^N(y)| 
			\leq C_t \Big(  \eps \sqrt{d} + 2 [\cT_N]\Big)^{\lambda} =: E_t^N(\eps) 
	\end{align*}
	for any $x, y \in Q_\eps(x_0)$, for a suitable $t$-dependent constant $C_t < \infty$ and $\lambda \in (0,1]$.
	It follows that 
	\begin{align}
		\label{eq:sandwich}
			\bigg(\sup_{Q_\eps(x_0)}  \rho_t^N \bigg) 
				- E_t^N(\eps) 
		\leq 
			\avint_{Q_\eps(x_0)} \rho_t^N \dd \bfm
		\leq 
			\bigg( \inf_{Q_\eps(x_0)} \rho_t^N \bigg)
				+ E_t^N(\eps). 
	\end{align}
	Taking into account that $r_t^N$ is a probability density, 
	it follows from the H\"older bound \eqref{eq:Holder} 
	that the famility $(\rho_t^N)_{N \geq 1}$ 
	is uniformly bounded in $L^\infty(\Omega, \bfm)$. 
	Hence, the Banach-Alaoglu Theorem yields 
	the existence of a subsequential weak$^*$-limit 
	$\rho_t \in L^\infty(\Omega, \bfm)$.
	Since $\bW(\mu_t^N, \mu_t) \to 0$, 
	we infer that $\mu_t = \rho_t \bfm$ and 
	$\int_{Q_\eps(x_0)} \rho_t^N \dd \bfm \to \mu_t(Q_\eps(x_0))$.
	Therefore, \eqref{eq:sandwich} yields 
	\begin{align*}
		\bigg( \limsup_{N \to \infty}
				\sup_{Q_\eps(x_0)} 
					\rho_t^N 
		\bigg)
					- C_t \big(  \eps \sqrt{d} \big)^{\lambda}
			\leq 
				\frac{\mu_t(Q_\eps(x_0))}{\bfm(Q_\eps(x_0))}
			\leq
		\bigg( \liminf_{N \to \infty}
				\inf_{Q_\eps(x_0)} 
					\rho_t^N 
		\bigg)	
					+ C_t \big(  \eps \sqrt{d} \big)^{\lambda}.
	\end{align*}
	Passing to the limit $\eps \to 0$ we obtain 
	\begin{align*}
		\lim_{\eps \to 0} 
			\liminf_{N \to \infty} 
				\sup_{x \in Q_\eps (x_0)} 
					\rho_t^N(x)
	\leq 
			\rho_t(x_0) 
	\leq 
		\lim_{\eps \to 0} 
				\limsup_{N \to \infty} 
				\inf_{x \in Q_\eps (x_0)} 
					\rho_t^N(x), 
	\end{align*}
	which is the desired result \eqref{eq:pc}.

	Therefore, we can apply Theorem \ref{thm:lower_bounds}\eqref{item:speed} to obtain the desired inequality 
	\begin{align*}	
		\liminf_{N \to \infty} 
			\cA_N^*(m_t^{N,\alpha}, e_N) 
		\geq 
			\bA^*(\mu_t^\alpha, \eta),
	\end{align*}	
	which implies that condition (ii) of Proposition \ref{prop:evolGammaliminf} is satisfied.

	\smallskip

	\noindent \textbf{Step 3.} \emph{Weak convergence of the time derivatives.}

	\smallskip	

	In order to apply Proposition \ref{prop:evolGammaliminf} we will now show that the sequence of time derivatives $\dot m^N$ is weakly convergent in $L^2\big((\delta, T); L^2(\Omega, \bfm)\big)$.

	Indeed, by self-adjointness of the discrete generator $\cL_N$ in $L^2(\cT_N, \pi_N)$ we have
	\begin{align*}
			\| \dot r_t^N \|_{L^2(\cT_N, \pi_N)}
		=   \| \cL_N r_t^N \|_{L^2(\cT_N, \pi_N)}
		\leq (t - \delta/2)^{-1} 
			\| r_{\delta/2}^N \|_{L^2(\cT_N, \pi_N)}
	\end{align*}
	for any $t > \delta/2$; see, e.g., \cite[Theorem~7.7]{brezis2010functional}. 
	Moreover, from \eqref{eq:ultracontractivity_discrete} we infer that 
	\begin{align*}
		\| r_t^N \|_{L^{\infty}(\cT_N, \pi_N)} 
			\lesssim 1 \vee t^{-\frac{d}{2} } 
	\end{align*}
	for $t >0$.
	As $\delta < 1$, it follows from these bounds that 
	\begin{align*}
		\int_{\delta}^T 
			 \| \rho_t^N \|_{L^2(\Omega, \bfm)}^2 
		\dd t  
	\lesssim
		T \delta^{-d}
	\tand
		\int_{\delta}^T 
			 \| \dot \rho_t^N \|_{L^2(\Omega, \bfm)}^2 
		\dd t  
	\lesssim 
		T \delta^{-(d + 1)}.
	\end{align*}
	The Banach-Alaoglu theorem implies that 
	any subsequence of $(\rho^N)_N$ has a subsequence 
	converging weakly in 
	$H^1 \big( (\delta,T); L^2(\Omega, \bfm) \big)$. 
	Since $\bW(\mu_t^N, \mu_t) \to 0$, 
	we infer that $\rho^N \weakto \rho$ in $H^1 \big( (\delta,T); L^2(\Omega, \bfm) \big)$, and $\rho_t = \frac{\dd \mu_t}{\dd \fm}$, as desired.

	Applying Proposition \ref{prop:evolGammaliminf} with 
	$\vphi_N(t) := \dot \rho_t^N $ and 
	$\vphi(t) := \dot \rho_t$, 
	we obtain 
	\begin{align*}
		\int_\delta^T 
			\bA^*(\mu_t^\alpha, \dot \mu_t) 
		\dd t 
	\leq 
		\liminf_{N \to \infty} 
		\int_\delta^T 
			\cA_N^*\big( m_t^{N,\alpha}, \dot m_t^N \big) 
		\dd t.
	\end{align*}

	\smallskip

	\noindent \textbf{Step 4.} \emph{Removal of the regularisation.}

	\smallskip

	Using the weak convergence $\mu_t^\alpha \weakto \mu_t$ as $\alpha \to 0$ and the weak lower-semicontinuity of $\bA^*(\cdot, \dot \mu_t)$, an application of Fatou's lemma yields
	\begin{align*}
		\int_\delta^T 
		\bA^*(\mu_t, \dot \mu_t) 
	\dd t 
\leq
	\liminf_{\alpha \to 0} 	
	\liminf_{N \to \infty} 
	\int_\delta^T 
		\cA_N^*\big( m_t^{N,\alpha}, \dot m_t^N \big) 
	\dd t.
	\end{align*}
	By convexity, we obtain
	\begin{align*}
		\cA_N^*\big( m_t^{N,\alpha}, \dot m_t^N \big) 
		\leq (1- \alpha)
			\cA_N^*\big( m_t^N, \dot m_t^N \big) 
		+  \alpha
			\cA_N^*\big( \pi_N, \dot m_t^N \big).
	\end{align*}
	We claim that
	$
		A := \sup_N \sup_{t \geq \delta} \cA_N^*\big( \pi_N, \dot m_t^N \big) < \infty
	$.
	Indeed, in view of the self-adjointness of the discrete generator $\cL_N$ and the ultracontractivity bound \eqref{eq:ultracontractivity_discrete}, we infer that
	\begin{align*}
		  \cA_N^*(\pi_N, \dot m_t^N)
		= \cA(\pi_N, r_t^N)
		= \cE_N(r_t^N)
		& \leq t^{-1} 
			\| r_t^N \|_{L^2(\cT_N, \pi_N)}^2
		\leq C t^{-1} 
		\big( 1 \vee t^{-d} \big),
	\end{align*}
	which yields the claim.
	Consequently, we obtain
	\begin{align*}
		\int_\delta^T 
		\bA^*(\mu_t, \dot \mu_t) 
	\dd t 
\leq
	\liminf_{N \to \infty} 	
	\int_\delta^T 
		\cA_N^*\big( m_t^N, \dot m_t^N \big) 
	\dd t.
	\end{align*}
	
	The final result follows by passing to the limit $\delta \rightarrow 0$.

\smallskip

\emph{(iii)}: This follows immediately by combining the inequalities from \emph{(ii)}. 
\end{proof}

\section{Mosco convergence of discrete energies: proof strategy} \label{section:mosco}

In this section we give a sketch of the proof of the Mosco convergence of the discrete energy functionals (Theorem \ref{thm:Mosco}).
This result is a key tool in the proof of evolutionary $\Gamma$-convergence; cf.~Section \ref{section:evolut_gamma_Wass}.  
Let us first recall the relevant definitions.

\begin{definition}[$\Gamma$- and Mosco convergence]	\label{def:Gamma}
	Let $\cF, \cF_N : \cX \rightarrow \R \cup \{ +\infty\}$ be functionals  defined on a complete metric space $\cX$. 
	The sequence $(\cF_N)_N$ is said to be {$\Gamma$-convergent} to $\cF$ if the following conditions hold:
\begin{enumerate}[(i)]
	\item For every sequence $(x_N)_N \subseteq \cX$ such that $x_N \to x \in \cX$ we have the \emph{liminf inequality}
	\begin{align}
		\liminf_{N \rightarrow \infty} \cF_N(x_N) \geq \cF(x).
	\end{align}
	\item For every $\bar x \in \cX$ there exists a \emph{recovery sequence} $(\bar x_N)_N \subseteq \cX$, i.e., $x_N \to x$ and  
	\begin{align}	
		\limsup_{N \rightarrow \infty} \cF_N(x_N) \leq \cF(x).
	\end{align}
\end{enumerate}
	If $\cX$ is a complete topological vector space, we say that $(\cF_N)_N$ is {Mosco convergent} to $\cF$ if the same conditions hold, with the modification that \emph{weakly} convergent sequences are considered in the liminf inequality.
\end{definition}

We use the notation $\cF_N \xrightarrow[]{\Gamma} \cF$ and   $\cF_N \xrightarrow[]{M} \cF$ to denote $\Gamma$- and Mosco convergence.

Let us now fix the setup, which remains in force throughout Sections \ref{section:mosco}, \ref{section:moscop}, and \ref{section:compactness_representation}.
Consider a family of $\zeta$-regular meshes $(\cT_N)_N$ with $[\cT_N] \to 0$ as $N \to \infty$. 
We then consider a measure $\mu \in \cP(\bOmega)$, 
and let  $\upsilon \in L^1(\Omega)$ be its density with respect to the Lebesgue measure.
At the discrete level we consider measures $m_N \in \cP(\cT_N)$.
We define the corresponding energy functionals $\cF_N$, $\tbF_N$, and $\bF_\mu$ as in Section \ref{section:statements}. 
The goal is to prove the Mosco convergence in $L^2(\Omega)$ of $\tbF_N$ to $\bF_\mu$ as $N \to \infty$ under the assumptions \eqref{eq:lb}, \eqref{eq:ub}, and \eqref{eq:pc}.
 
Our strategy is based on a compactness and representation procedure, following ideas from \cite{Alicandro2004}.
A key ingredient in the proof is 
a representation result from \cite[Theorem 2]{Bouchitte2002}, 
for which we need to perform a localisation procedure. 
Let $\cO(\Omega)$ be the collection of all open subsets of $\Omega$.
For $A \in \cO(\Omega)$ we then introduce the functionals
$\cF_\cT : L^2(\cT, \pi_\cT) \times \cO(\Omega) 
\to [0, +\infty)$ 
by
\begin{align*}
	\cF_\cT(f,A)
	&:= 
		\frac14
		\sum_{ 
		 	\substack{ 
		 		K,L \in \cT \vert_A
		 	}
		} 
		\big( f(K) - f(L) \big)^2 
		U_{KL}
		\frac{|\Gamma_{KL}|}{d_{KL}},
\end{align*}
where, for any subset $A \subseteq \Omega$,
\begin{align}	\label{eq:defTA}
	\cT\vert_A := 
		\left\{ K \in \cT \suchthat
			 \overline{K} \cap A \neq \emptyset \right\}
\end{align}
and $U_{KL}$ is as in Section \ref{section:statements}.
The corresponding embedded functional 
$\tbF_\cT : L^2(\Omega) \times \cO(\Omega) \to [0, + \infty]$ 
is given by
\begin{align*}
	\tbF_\cT(\vphi,A)
	&:= \begin{cases}
			\cF_\cT(P_{\cT} \vphi, A)
				& \text{if }\vphi \in \PCcT \\
			+\infty 
				& \text{otherwise},
		\end{cases}
\end{align*}
where $P_\cT$ is the projection defined in \eqref{eq:PT}.

The proof of Theorem \ref{thm:Mosco} consists of the following steps:

\begin{enumerate}[(Step 1)]
\item \label{item2} We show first, as in \cite[Proposition 3.4]{Alicandro2004}, that any subsequential $\Gamma$-limit point $\bF(\cdot, A)$ of the sequence $\big(\tbF_N(\cdot,A)\big)_N$ is only finite on $H^1(\Omega)$. 
This result is a prerequisite for performing Step \ref{item3}.
We also show that $\Gamma$-convergence implies Mosco convergence in this situation.

\item \label{item1}
For any subsequential $\Gamma$-limit point $\bF(\cdot, A)$, we prove an inner regularity result. 
Using this result, we can apply a compactness result \cite[Theorem 10.3]{Braides1998} to infer that there exists a subsequence, such that,  for any $A \in \cO(\Omega)$, the functionals $\big(\tbF_N(\cdot,A)\big)_N$ $\Gamma$-converge to a limiting functional $\bF(\cdot, A)$. 

\item \label{item3} We prove the applicability of a representation theorem \cite[Theorem 2]{Bouchitte2002}, which allows us to deduce the following expression
\begin{align}	\label{eq:E0}
	\bF(\vphi)
	= \int_{\Omega} F(x,\vphi,\nabla \vphi) \dd x.
\end{align}
\item \label{item4} In view of the previous steps, it remains to 
show that $F(x, u,\xi) = \upsilon(x)|\xi|^2$.
\end{enumerate}

\smallskip

Steps \ref{item2} and \ref{item1} will be carried out in Section \ref{section:moscop}, while Steps \ref{item3} and \ref{item4} will be performed in Section \ref{section:compactness_representation}.

\section{Mosco convergence of the localised functionals}
\label{section:moscop}

In this section we perform Steps 1 and 2 of the proof strategy described above. 
As before, we consider a vanishing sequence of $\zeta$-regular meshes $(\cT_N)_N$
and a sequence of discrete measures $m_N \in \cP(\cT_N)$.
We will prove the following results.

\begin{theorem}[Regularity of $\Gamma$-limits]	\label{thm:Sob-finite}
Assume \eqref{eq:lb}. 
For $A \in \cO(\Omega)$, 
let $\bF(\cdot, A)$ be a subsequential $\Gamma$-limit 
of the sequence $\big( \tbF_N(\cdot,A) \big)_N$ 
in the $L^2(\Omega)$-topology.
Then $\bF(\vphi, A) = + \infty$ 
for any $\vphi \notin H^1(\Omega)$. 
Moreover, the subsequence is also convergent in the Mosco sense.
\end{theorem}

The proof of this result is contained in Section \ref{subsec:reg-finite-energy} and relies on an $L^2$-H\"older continuity result (Proposition \ref{prop:sobolev_limit}).

\begin{theorem}[Localised Mosco compactness] 
\label{thm:mosco-pre_exist}
Assume \eqref{eq:lb} and \eqref{eq:ub}.
There exists a subsequence of $(\tbF_N)_N$ such that, 
for any $A \in \cO(\Omega)$, 
the sequence $\big( \tbF_N(\cdot,A) \big)_N$ 
is Mosco convergent in $L^2(\Omega)$-topology.
\end{theorem}

The proof of this result is contained 
in Section \ref{section:sobolev} 
and relies on an inner regularity result 
(Proposition \ref{prop:inner_regularity}). 
The latter result will be proved 
using a Sobolev upper bound 
(Proposition \ref{prop:sobolev_bound}).

\subsection{Regularity of finite energy sequences}\label{subsec:reg-finite-energy}

In this subsection we prove that 
any subsequential $\Gamma$-limit $\bF$ 
of the sequence $\big( \tbF_N(\cdot,A) \big)_N$  
is only finite on Sobolev maps, 
which allows us to work with Theorem \ref{thm:representation}. 
A corresponding result was proved on the cartesian grid in \cite[Proposition 3.4]{Alicandro2004}, using affine interpolations of vector fields that are not available in the present context.

For $h \in \R^d$ we write $K \simh L$ if $\overline{K} \cap (\overline{L} + h) \neq \emptyset$.

\begin{lemma}[Existence of good paths]	\label{lemma:paths}
Let $\cT$ be a $\zeta$-regular mesh. Then there exists a family of paths $\{ \gamma_{KL} \}_{ K, L \in \cT}$, where
\begin{align*}
 	\gamma_{KL} = \{\gamma_{KL}(i) \suchthat i = 0, \ldots, n_{KL}  \},
	 \quad	
			  K = \gamma_{KL}(0) \sim \gamma_{KL}(1) \sim \ldots \sim \gamma_{KL}(n_{KL}) = L,
\end{align*}
such that the following properties hold:
\begin{enumerate}
\item For all $K, L \in \cT$ we have
\begin{align}  \label{eq:lemma_paths_distancebound}
	n_{KL} 
		\lesssim 
	\frac{|x_{K} - x_{L}|}{[\cT]} 
\tand
	\sum_{i=0}^{n_{KL}} 
	|x_{\gamma_{KL}(i)} - x_{\gamma_{KL}(i+1)}| 
		\lesssim 
	|x_{K} - x_{L}|;
\end{align}
\item For any $h \in \R^d$ and $M, N \in \cT$ with $M \sim N$ we have
\begin{align} \label{eq:lemma_paths_countingcouples}
	\# \left\{ (K,L) \in \cT^2 \suchthat K \simh L, \;  \{ M, N \} \subset \gamma_{KL} \right\} \lesssim 1 \vee \frac{|h|}{[\cT]}.
\end{align}
\end{enumerate}
The implied constants depend only on $\Omega$ and $\zeta$.
\end{lemma}

\begin{proof}
Part (1) has been proved in \cite[Lemma 2.12]{gladbach2018}, 
so we focus on (2). 

Fix $h \in \R^d$ and $M, N \in \cT$ with $M \sim N$.
Without loss of generality we may assume that $x_M = 0$ and $h$ is parallel to the $d$-th unit vector in $\R^d$. 
Let $\cS$ be the set whose cardinality we would like to bound, and let $\cS_1$ be the collection of all $K \in \cT$ 
such that $(K,L) \in \cS$ for some $L \in \cT$. 

We claim that
\begin{align} \label{eq:proof_lemmabounds}
	\bigcup_{K \in \cS_1} K \subset \Cyl({r}, {\ell})
\end{align} 
for some $r \lesssim [\cT]$ and $\ell \lesssim |h| + [\cT]$. Here, $\Cyl(r, \ell)$ denotes the cylinder of radius $r > 0$ and height $2 \ell > 0$, i.e.,
\begin{align*}
	\Cyl(r, \ell) := 
		\big\{ 
			v \in \R^d 
				\suchthat 
				v^* \in B_r^{d-1}, \; 
				v_d \in [-\ell, \ell] 
		\big\}.
\end{align*}
where $B_r^{d-1}$ denotes the closed ball of radius $r$ around the origin in $\R^{d-1}$. 

Indeed, by 
the construction in \cite{gladbach2018}, 
$M \cup N$ is contained in the cylinder of radius $2[\cT]$, 
whose central axis is obtained by extending the line segment 
between $x_K$ and $x_L$ by a distance $[\cT]$ in both directions, 
for all cells $K,L \in \cT$.
The claim follows using the fact that $K \simh L$.
 
Next we use a simple volume comparison. 
Using $\zeta$-regularity, it follows that
\begin{align}	\label{eq:proof_measureS1}
	\Leb^d\Big( 
				\bigcup_{K \in \cS_1} K 
		  \Big) 
	= 		\sum_{K \in \cS_1} 
				\Leb^d(K) 
	\gtrsim  [\cT]^d (\# \cS_1),
\end{align}
where $\#\cS_1$ denotes the cardinality of $\cS_1$.
Combining \eqref{eq:proof_lemmabounds} and \eqref{eq:proof_measureS1} we infer that
$
	\#\cS_1 \lesssim 1 \vee \frac{|h|}{[\cT]}.
$

To conclude the proof, it remains to show that $\#\cS \lesssim \#\cS_1$. 
To see this, note that for every $K \in \cS_1$, there exists a universally bounded number of cells $L \in \cT$ such that $(K,L) \in \cS$. 
This is due to the fact that if $L, L' \in \cT$ are such that $(K,L), (K,L') \in \cS$, 
we deduce that $d_{L,L'} \lesssim [\cT]$
by the triangle inequality.
The desired result follows from this observation by $\zeta$-regularity.
\end{proof}

The following lemma is the crucial estimate needed to deduce $L^2$-strong compactness of sequences with bounded energy. 
A similar result has been obtained in dimension $d = 2, 3$ in \cite[Lemma~3.3]{eymard2000finite} with bounds in terms of discrete Sobolev norms. 
 
\begin{lemma}[$L^2$-H\"{o}lder continuity]	\label{lem:sobolev_limit}
Assume \eqref{eq:lb}. 
Fix $A \in \cO(\Omega)$ and set $A_\delta := \{ x \in A \suchthat {\rm dist}(x,\partial \Omega) > \delta\}$ for $\delta > 0$.
Let $\cT$ be a $\zeta$-regular mesh, let $f \in L^2(\cT\vert_A)$ and define $\vphi := \Q_\cT f \in L^2(A)$.
For any $h \in \R^d$ we have the $L^2$-bound
\begin{align}	\label{eq:bound_Lipschitz_T}
	\| \tau_h \vphi - \vphi \|_{L^2(A_{|h|})}^2 	\lesssim
		\frac{|h|}{\underline{k}}  
		\Big( |h| \vee [\cT] \Big)
		\cF_\cT(f,A),
\end{align}
where $\tau_h \vphi (\cdot):= \vphi(\cdot - h)$, and $\underline{k} > 0$ is the lower bound in \eqref{eq:lb}. 
\end{lemma}

\begin{proof}
For any $h \in \R^d$ we have 
\begin{align}
\label{eq:L2-est}
	\| \tau_h \vphi - \vphi \|_{L^2(A_{|h|})}^2 
	& = \int_{A_{|h|}} 
		\big( \vphi(x-h) - \vphi(x) \big)^2 
		\dd x
       \leq \sum_{K,L \in \cT \vert_A} 
   			| C_{KL} | \big( f(L) - f(K) \big)^2 
		,
\end{align}
where $C_{KL} = \{x \in K : x - h \in L \}$. 
For $K, L \in \cT \vert_A$ we use Lemma \ref{lemma:paths} and the Cauchy-Schwarz inequality to write
\begin{align}	\label{eq:bound_long_interactions}
	\big( f_N(K) - f_N(L) \big)^2 
		\leq n_{KL} \sum_{i=1}^{n_{KL}} 
			\big( f_N(K_{i-1}) - f_N(K_i)\big)^2,
\end{align}
where $K = K_0 \sim K_1 \sim \ldots \sim K_{n_{KL}} = L$, and $n_{KL} \lesssim \frac{d_{KL}}{[\cT]}$.
Observe that $d_{KL} \lesssim [\cT] \vee |h|$ whenever $C_{KL} \neq \emptyset$.

To estimate the measure of $C_{KL}$, we pick a hyperplane $H$ that separates $K$ and $L$ (which exists by the Hahn-Banach theorem, in view of the convexity of the cells). 
By construction, $C_{KL}$ is contained in the strip between $H$ and $H + h$. 
Moreover, we have $C_{KL} \subseteq K$, which means that $C_{KL}$ is contained in a ball of radius $\lesssim [\cT]$. 
Combining these two facts, we infer that $|C_{KL}| \lesssim [\cT]^{d-1} |h|$, hence $|C_{KL}| \lesssim [\cT]^{d-1}\big(|h| \wedge [\cT]\big)$ by $\zeta$-regularity.

Putting these estimates together, we obtain 
\begin{align}
\label{eq:S-bound}
	|C_{KL}| \big( f(K) - f(L) \big)^2 
	\lesssim  [\cT]^{d-1} |h|
		\sum_{i=1}^{n_{KL}} 
			\big( f(K_{i-1}) - f(K_i) \big)^2.
\end{align}
Let $\alpha_{KL}$ denote the left-hand side in 
\eqref{eq:lemma_paths_countingcouples}. 
Using \eqref{eq:L2-est} and \eqref{eq:S-bound} we find that
\begin{align*}
\| \tau_h \vphi - \vphi \|_{L^2(A_{|h|})}^2 
&	\lesssim  [\cT]^{d-1} |h|
		\sum_{ 
			\substack{ 
				K,L \in \cT \vert_A \\ 
				L \sim K
			}
		} 
 \alpha_{KL} \big( f(L) - f(K) \big)^2.
\end{align*}

On the other hand, in view of the $\zeta$-regularity and the assumption \eqref{eq:lb}, we have
\begin{align*}
	\cF_\cT(f,A) \gtrsim 
	\underline k[\cT]^{d-2}
	\sum_{ 
		\substack{ 
			K,L \in \cT \vert_A \\ 
			L \sim K
		}
	}
			\big( f(K) - f(L) \big)^2.
\end{align*}
The desired result follows, since $\alpha_{KL} \leq 1 \vee \frac{|h|}{[\cT]}$ by Lemma \ref{lemma:paths}.
\end{proof}

The compactness result now follows easily.

\begin{proposition}[Compactness]	\label{prop:sobolev_limit}
Fix $A \in \cO(\Omega)$ and assume that the lower bound \eqref{eq:lb} holds.
Let $(\cT_N)_N$ be a vanishing sequence of $\zeta$-regular meshes. Let $f_N \in L^2(\cT_N\vert_A)$ be such that  
\begin{align*}
	\alpha:= \sup_{N \in \N} \cF_N(f_N,A)  < +\infty,
\end{align*}
and define $\vphi_N := \Q_N f_N \in L^2(A)$.
Then the sequence $(\vphi_N)_N$ is relatively compact in $L^2(A)$. Moreover, any subsequential limit $\vphi$ belongs to $H^1(A)$ and satisfies
\begin{align*}
  \| \nabla \vphi \|_{L^2(A)} \lesssim \sqrt{\frac{\alpha}{k}}.
\end{align*} 
\end{proposition}

\begin{proof}
The $L^2$-compactness follows from \eqref{eq:bound_Lipschitz_T} in view of the Kolmogorov-Riesz-Frech\'{e}t theorem \cite[Theorem 4.26]{brezis2010functional}. 
Let $\vphi$ be any subsequential limit point of $\vphi_N$ as $[\cT_N] \rightarrow 0$. 
Another application of \eqref{eq:bound_Lipschitz_T} yields,  for any $h \in \R^d$ and $\delta > 0$,
\begin{align*}
	\| \tau_h \vphi - \vphi \|_{L^2(A_\delta)}^2 
		& = \lim_{N \to \infty}
		 \| \tau_h \vphi_N - \vphi_N\|_{L^2(A_\delta)}^2 
	\lesssim \frac{\alpha}{k} |h|^2,
\end{align*} 
which implies that $\vphi \in H^1(A)$ by the characterisation of $H^1(A)$ as the space of functions which are Lipschitz continuous in $L^2$-norm (see, e.g., \cite[Proposition 9.3]{brezis2010functional}).
\end{proof}

\begin{proof}[Proof of Theorem \ref{thm:Sob-finite}]
Proposition \ref{prop:sobolev_limit} shows that $\vphi \in H^1(\Omega)$ whenever $\bF(\vphi) < \infty$. 
It also follows from Proposition \ref{prop:sobolev_limit} that every $L^2$-weakly convergent sequence $\vphi_N = \Q_N f_N$ with bounded energy $\sup_N \cF_N(f_N,A) < +\infty$ converges strongly in $L^2$. 
Therefore, Mosco and $\Gamma$-convergence are equivalent in this situation.
\end{proof}

\subsection{Sobolev bound and inner regularity}	\label{section:sobolev}

This part focuses on a Sobolev upper bound for subsequential $\Gamma$-limit functionals, which will be useful in Proposition \ref{prop:inner_regularity} and in Theorem \ref{thm:representation} below.

\begin{proposition}[Sobolev upper bound] \label{prop:sobolev_bound}
Assume \eqref{eq:ub} and let $A \in \cO(\Omega)$.
For any subsequential $\Gamma$-limit  $\bF(\cdot, A)$
of the sequence $\big( \tbF_N(\cdot,A) \big)_N$ 
in the $L^2(\Omega)$-topology,
we have the Sobolev upper bound
\begin{align}	\label{eq:sobolev_bound_prop}
	\bF(\vphi,A) 
		\lesssim 
	\bar{k} \int_A |\nabla \vphi |^2 \dd x
\end{align}
for any $\vphi \in H^1(\Omega)$.
\end{proposition}

Here and in the proof, the implied constants depend only on $\Omega$ and the regularity parameter $\zeta$.

\begin{proof}
Let us first prove \eqref{eq:sobolev_bound_prop} for $\vphi \in C_c^{\infty}(\R^d)$. 
For $N \in \N$, define $f_N : \cT_N \to \R$ by $K \in \cT_N$, define
\begin{equation*}
	f_N(K):= \vphi(x_K) \quad\text{for } K \in \cT_N.
\end{equation*}
Write $\nu_{KL} := \frac{x_K - x_L}{d_{KL}}$. 
By smoothness of $\vphi$ and $\sigma$, we have
\begin{equation*}	
\eps_N := 
	\sup_{K, L \in \cT_N} 
		\bigg\vert	
			\bigg( 
				\frac{ f_N(K) - f_N(L) }{ d_{KL} } 
			\bigg)^2 
			 -  
			\big(
				\nabla \varphi(x_K) \cdot \nu_{KL} 
			\big)^2 
		\bigg\vert \to 0. 
\end{equation*} 
Using this estimate, assumption \eqref{eq:ub}, and the $\zeta$-regularity, we obtain 
\begin{align*}
	\tbF_N(Q_N f_N,A) 
	& =
		\frac14
			\sum_{K,L \in \cT_N\vert_A}  
				\bigg( \frac{ f_N(K) - f_N(L) }{d_{KL}} \bigg)^2  	
				U_{KL}
				d_{KL} |\Gamma_{KL}|  
	\\& 
	\lesssim    
	 	\overline{k}
	 		\sum_{K \in \cT_N\vert_A}  
			 	\Big(
					 \big(
						\nabla \varphi(x_K) \cdot \nu_{KL} 
					 \big)^2
					+ \eps_N	
				\Big)
					\bigg( 
						\sum_{L: L \sim K} d_{KL} |\Gamma_{KL}
					\bigg)
	\\& 
	\lesssim  
	 	\overline{k}	 
	 		\sum_{K \in \cT_N\vert_A}  
				\Big(
					| \nabla \vphi(x_K) |^2  	
				+ \eps_N	
				\Big)	
					|K|.					
  \end{align*}
The smoothness of the function $|\nabla \vphi|^2$ and the identity $\sum_{K \in \cT_N} |K| = |\Omega|$ now yield
\begin{align*}
	\limsup_{N \to \infty} 
	\tbF_N(Q_N f_N,A) 
\lesssim 
\bar{k} \int_A 
			|\nabla \vphi|^2 
		\dd x.
\end{align*}
Since $Q_N f_N$ converges to $\vphi$ in $L^2(A)$, the  $\Gamma$-convergence of $\tbF_N(\cdot, A)$ to $\bF(\cdot, A)$ yields the desired bound \eqref{eq:sobolev_bound_prop}.

It remains to extend the result to $H^1(\Omega)$ by a density argument.
Indeed, for any $\vphi \in H^1(\Omega)$ 
there exists a sequence 
$(\vphi_i)_i \subseteq C_c^{\infty}(\R^d)$ 
such that 
$\vphi^i \to \vphi$ in $H^1(\Omega)$. 
As $\bF(\cdot,A)$ is lower semicontinuous in $L^2(\Omega)$, 
we can apply \eqref{eq:sobolev_bound_prop} to $\vphi_i$ to obtain
\begin{align*}
	\bF(\vphi,A) 
		\leq \liminf_{i \to \infty}
				 \bF(\vphi_i,A) 
		\lesssim \bar{k}\liminf_{i \to \infty} 
				\int_{A} |\nabla \vphi_i |^2 \dd x
		= \bar{k} \int_{A} |\nabla \vphi |^2 \dd x,
\end{align*}
which shows \eqref{eq:sobolev_bound_prop} for $\vphi \in H^1(\Omega)$.
\end{proof}

\begin{remark}	\label{rem:limsup}
In the case where $m_N = \mathsf{P}_N (\rho \dd x)$ for a continuous density $\rho$, 
it is possible to prove the sharp upper bound $\bF \leq \bF_\mu$ by a similar argument with a bit more effort. 
However, we are not aware of a simple argument for the corresponding liminf inequality. 
Therefore, we pass through the compactness and representation scheme, which yields the sharp upper bound as a byproduct.
\end{remark}

We now focus on the inner regularity of subsequential $\Gamma$-limit functionals. 
We will prove something slightly stronger than the classical inner regularity, 
namely, an inner approximation result with sets of Lebesgue measure $0$.
This sharpening will be useful in the proof of the locality in Proposition \ref{prop:locality} below. 

For any $A,B \subset \Omega$, 
we write $A \ssubset B$ 
as a shorthand for $A$ being relatively compact in $B$.

\begin{proposition}[Inner regularity] \label{prop:inner_regularity}
Assume \eqref{eq:ub}. 
For $A \in \cO(\Omega)$, 
let $\bF(\cdot, A)$ be a subsequential $\Gamma$-limit  
of the sequence $\big( \tbF_N(\cdot,A) \big)_N$
in the $L^2(\Omega)$-topology.
Then the function $A \mapsto \bF(\vphi,A)$ 
is inner regular on $\cO(\Omega)$, 
i.e.,
\begin{align}
	\label{eq:inner-regularity}
	\sup_{ 
			\substack{ 
				A' \Subset A \\ 
				\Leb^d(\partial A') =0
			}
		}
		\bF(\vphi,A')
	 = \sup_{A' \Subset A} \bF(\vphi,A') 
	 = \bF(\vphi,A).
\end{align}
for any $\vphi \in H^1(\Omega)$ and $A \in \cO(\Omega)$.
\end{proposition}

\begin{proof}
Fix $\vphi \in H^1(\Omega)$	and $A \in \cO(\Omega)$.
It immediately follows from the definitions that \eqref{eq:inner-regularity} holds with ``$\leq$'' (twice) instead of ``$=$''. 
It thus suffices to prove that
\begin{align*}
	\bF(\vphi,A)
	\leq \sup_{ 
	\substack{ 
		A' \Subset A \\ 
		\Leb^d(\partial A') =0
	}
}
\bF(\vphi,A').
\end{align*}
We adapt the proof for the cartesian grid as given in \cite[Proposition 3.9]{Alicandro2004}. 

Fix $\delta > 0$ 
and consider a non-empty set $A'' \in \cO(\Omega)$ 
such that $A'' \ssubset A$ and
\begin{align*}
	\int_{A \setminus \overline{A''}} 
		|\nabla \vphi|^2 
	\dd x 
	< \delta.
\end{align*} 
Let $\eps_N := \Q_N e_N$ be a recovery sequence for $\bF(\vphi, A \setminus \overline{A''})$, i.e., 
\begin{align}	\label{eq:proof_inner_deltaerr}
	\eps_N \rightarrow \vphi \text{ in } L^2(\Omega)
	\quad \text{and} \quad
	\limsup_{N \to \infty} 
		\cF_N(e_N,A \setminus \overline{A''})
		\leq 
		\bF(\vphi, A \setminus \overline{A''}) 
		\lesssim \bar{k} \delta,
\end{align}
where the last bound is a consequence of Proposition \ref{prop:sobolev_bound}.

Take $A' \in \cO(\Omega)$ such that $A'' \ssubset A' \ssubset A$ and $\Leb^d(\partial A')=0$. 
Note that this can always be done, since one can pick a compact set $K$ satisfying $A'' \subset K \ssubset A$, and then choose $A'$ as the union of any finite open cover of $K$ by balls whose closures are contained in $A$.
Let $\vphi_N := Q_N f_N$ be a recovery sequence for $\bF(\vphi, A')$, so that
\begin{align}
	\label{eq:limsup}
	\vphi_N \rightarrow \vphi \text{ in } L^2(\Omega) 
	\quad \text{and} \quad
	\limsup_{N \to \infty} 
		\cF_N(f_N, A') 
		\leq \bF(\vphi, A').
\end{align}  
Fix $M \in \N$ and suppose that $[\cT_N] < \frac{1}{5(M+1)}$.  
Define $A'' \subset A_1 \subset A_2 \subset \ldots \subset A_{5(M+1)} \subset A'$ by
\begin{align*}
	A_j := 	
		\left\{ 
			x \in A' \suchthat	
				d(x,A'') 
				< 
				\frac{j}{5(M+1)} d\big((A')^c, A''\big)  
		\right\}.
\end{align*} 
Moreover, for $i \in \{1, \ldots, M \}$ we consider a cutoff function $\rho_i \in C^{\infty}(\R^d)$ satisfying 
\begin{align}	\label{eq:cutoff}
	\rho_i|_{A_{5i+2}} = 1, \quad 
	\rho_i|_{\Omega \setminus A_{5i+3}} = 0, \quad 
	0 \leq \rho_i \leq 1, \quad
	|\nabla \rho_i | \lesssim M.
\end{align}
Set $r_N^i(K) := \rho_i(x_K)$ for $K \in \cT_N$,  and define  
\begin{align*}
	f_N^i := r_N^i f_N + (1 - r_N^i) e_N, 
\quad \text{so that }			
	\vphi_N^i := \Q_N f_N^i	\to \vphi
\end{align*}
as $N \to \infty$, uniformly for $i \in \{1, \ldots, M\}$.
As $[\cT_N]<\frac{1}{5(M+1)}$, we have by \eqref{eq:cutoff},
\begin{align}	\label{eq:cutoff_props-alt}
	f_N^i	\equiv f_N	\; 
		\text{on } 
		\cT_N \vert_{A_{5i+1}}, \quad  
	f_N^i	\equiv e_N	\; 
		\text{on } 
		\cT_N \vert_{(A \setminus \overline{A_{5i+4}})}.
\end{align}
Using these identities and the inclusions $A_{5i+1} \subset A'$ 
and $A'' \subset A_{5i+4}$ we obtain    
\begin{align}	\label{eq:proof_inner_upbound-alt}
	\begin{aligned}
		\cF_N(f_N^i,A) 
			& \leq \cF_N(f_N^i,A_{5i+1})
				 + \cF_N(f_N^i,A_{5(i+1)} 
							 \setminus \overline{A_{5i}})
				 + \cF_N(f_N^i,A \setminus \overline{A_{5i+4}})    
		 \\	&\leq \cF_N(f_N,A') 
				 + \cF_N(f_N^i,A_{5(i+1)} 
							 \setminus \overline{A_{5i}})
				 + \cF_N(e_N,A \setminus \overline{A''}).
			\end{aligned}
	\end{align}

To estimate the middle term, let $\nabla g(K,L) := g(L) - g(K)$ denote the discrete derivative and observe that
\begin{align*}
	\nabla f_N^i(K,L) 
	   & = r_N^i(L) \nabla f_N(K,L) 
	   			+ \big( 1 - r_N^i(L) 	\big) \nabla e_N(K,L) 
	\\ & \quad 	+ \big( f_N(K) - e_N(K) \big) \nabla r_N^i(K,L)
\end{align*}
for any $K,L \in \cT_N$.
Consequently, 
\begin{align*}
	| \nabla f_N^i(K,L) |^2 
	&\lesssim |\nabla f_N(K,L)|^2 + |\nabla e_N(K,L) |^2 
	+ M^2 d_{KL}^2 |f_N(K) - e_N(K)|^2.
\end{align*}
Using this bound and the $\zeta$-regularity of the mesh, we obtain
\begin{align*}	
	&  \sum_{i=1}^M  
			\cF_N(f_N^i,A_{5(i+1)} 
			 			\setminus \overline{A_{5i}})
	\\  & \lesssim 
		\sum_{i=1}^M  
			\bigg(
		\cF_N(f_N, A_{5(i+1)} \setminus \overline{A_{5i}}) 
	  + \cF_N(e_N, A_{5(i+1)} \setminus \overline{A_{5i}})
		+ \bar{k} M^2 \| \vphi_N - \eps_N \|_{L^2(\Omega)}^2
			\bigg)
	\\& \leq	
		2 \Big( \cF_N(f_N,A' \setminus \overline{A''}) 
			+ \cF_N(e_N,A' \setminus \overline{A''}) \Big) 
	+ \bar{k} M^3 \| \vphi_N - \eps_N \|_{L^2(\Omega)}^2.
\end{align*}
Taking into account that that $\vphi_N, \eps_N \to \vphi$ in $L^2$, we can pass to the limsup as $N \to \infty$, using \eqref{eq:proof_inner_deltaerr}, \eqref{eq:limsup}, and Proposition \ref{prop:sobolev_bound}, to obtain
\begin{align*}
	\limsup_{N \to \infty}
	\sum_{i=1}^M  
	\cF_N( 	f_N^i,
			A_{5(i+1)} \setminus \overline{A_{5i}}
		)
	& \lesssim 
		\limsup_{N \to \infty}
				\cF_N(f_N, A') 
	+ 
		\limsup_{N \to \infty}
			\cF_N(e_N,A \setminus \overline{A''})
	\\& \lesssim 
			\bF(\vphi, A')
	+  
			\bF(\vphi, A \setminus \overline{A''}) 
		\\& \lesssim
	\overline{k}
	\int_A |\nabla \vphi |^2 \dd x.
\end{align*}
Using this bound and \eqref{eq:proof_inner_deltaerr},  \eqref{eq:limsup} once more, 
it follows from \eqref{eq:proof_inner_upbound-alt} that
\begin{align*}
	\limsup_{N \to \infty}
		\bigg(
			\frac{1}{M} \sum_{i=1}^M 
			\cF_N(f_N^i,A) 
		\bigg)
\leq  \bF(\vphi,A') 
 + C \overline{k} \bigg( 
	 \frac{1}{M}
			\int_A |\nabla \vphi |^2 \dd x
+  \delta
	\bigg).
\end{align*}
where $C < \infty$ depends only on $\Omega$, $\zeta$.

Clearly, for each $N$, there exists $i_N \in \{1, \ldots, M\}$ such that
\begin{align*}
	\cF_N(f_N^{i_N},A) 
	\leq 
		\frac{1}{M} \sum_{i=1}^M 
		\cF_N(f_N^i,A),
\end{align*}
Since $\sup_{1 \leq i \leq M} \| \vphi_N^i - \vphi \|_{L^2(\Omega)} \to 0$ as $N \to \infty$, we have $\vphi_N^{i_N} \to \vphi$ in $L^2(\Omega)$. 
Therefore, using the $\Gamma$-convergence we obtain
\begin{align*}
	\bF(\vphi,A) 
	\leq \liminf_{N \to \infty} 
		\cF_N(f_N^{i_N},A) 
	\leq \bF(\vphi,A') 
		+ C  \overline{k} \bigg( 
		\frac{1}{M}
			   \int_A |\nabla \vphi |^2 \dd x
   +  \delta
	   \bigg).
\end{align*}
As $\delta > 0$  and $M < \infty$ are arbitrary, this is the desired result.
\end{proof}

\begin{proof}[Proof of Theorem \ref{thm:mosco-pre_exist}] 
By Proposition \ref{prop:inner_regularity} and \cite[Theorem 10.3]{Braides1998}, 
there exists a subsequence such that, for any $A \in \cO(\Omega)$, 
the functionals $\big(\tbF_N(\cdot,A)\big)_N$ are $\Gamma$-converging in $L^2(\Omega)$-topology 
to a limit functional $\bF(\cdot,A)$. 
The fact that $\Gamma$-convergence implies Mosco convergence has already been observed in Theorem \ref{thm:Sob-finite}.
\end{proof}

\section{Representation and characterisation of the limit}	\label{section:compactness_representation}

We fix the same setup as in Section \ref{section:moscop}.
We thus consider a vanishing sequence of 
$\zeta$-regular meshes $(\cT_N)_N$
and a sequence of discrete measures $m_N \in \cP(\cT_N)$.

We show the following representation formula for the $\Gamma$-limits from Section \ref{section:moscop}:

\begin{theorem}[Representation of the $\Gamma$-limit] 
\label{thm:mosco-pre_rapp_char}
Assume \eqref{eq:lb} and \eqref{eq:ub}, and 
suppose that, for every $A \in \cO(\Omega)$, the functionals $\big(\tbF_N(\cdot,A)\big)_N$ are $L^2(\Omega)$-Mosco convergent to a functional $\bF(\cdot,A)$. 
Then the functional $\bF$ can be represented as
\begin{align}	\label{eq:rappF0}
	\bF(\vphi,A) = 
	\begin{cases}
		\displaystyle
		\int_A F(x,\vphi, \nabla \vphi) \dd x 
			& \text{for } \vphi \in H^1(\Omega), \\
		+\infty
			& \text{for } \vphi \in L^2(\Omega) \setminus H^1(\Omega),
	\end{cases}
\end{align}
for some measurable function $F: \Omega \times \R  \times \R^d \to [0,+\infty)$.
\end{theorem}

Combined with the following result, this will complete the proof of Theorem \ref{thm:Mosco}.

\begin{theorem}[Characterisation of $F$]	\label{thm:f_quadratic}
Assume \eqref{eq:lb}, \eqref{eq:ub}, and \eqref{eq:pc}.
Then the function $F : \Omega \times \R^d \to [0,+\infty)$ defined in Theorem \ref{thm:mosco-pre_rapp_char} is given by
\begin{align*}
	F(x,u,\xi) = |\xi|^2 \upsilon(x)
	\quad \forall x \in \Omega, \; 
	u \in \R, \; \xi \in \R^d.
\end{align*}
In particular, the sequence $\big(\tbF_N(\cdot,A)\big)_N$ is $L^2(\Omega)$-Mosco convergent to $\bF_\mu(\cdot,A)$.
\end{theorem}

To prove Theorem \ref{thm:mosco-pre_rapp_char}, 
we use a representation result 
for functionals on Sobolev spaces \cite{Bouchitte2002}.
In our application, we have $\bE(\cdot,A):= \bF(\cdot,A)$, where $\bF(\cdot,A)$ is a subsequential $\Gamma$-limit point of 
$\big( \tbF_N(\cdot,A) \big)_N$.

\begin{theorem}	\label{thm:representation}
Let $\bE: H^1(\Omega) \times \cO(\Omega) \rightarrow [0,+\infty]$ be a functional satisfying the following conditions:
\begin{enumerate}[(i)]

\item \label{it:local} \emph{locality}: 
$\bE$ is local, i.e., for all $A \in \cO(\Omega)$ 
we have $\bE(\vphi,A) = \bE(\psi,A)$ 
if $\vphi = \psi$ a.e. on $A$.

\item \label{it:subadditivity} \emph{measure property}: 
For every $\vphi \in H^1(\Omega)$ 
the set map $\bE(\vphi, \cdot)$ is the restriction of a Borel measure to $\cO(\Omega)$.

\item \label{it:Sobolev} \emph{Sobolev bound}:  
There exists a constant $c>0$ and $a \in L^1(\Omega)$ 
such that
\begin{align*}
	 \frac{1}{c} \int_A | \nabla \vphi|^2 \dd x 
	 	\leq \bE(\vphi, A) 
		\leq c \int_A 
			\big(a(x) + |\nabla \vphi|^2 \big)
				\dd x
\end{align*}
for all $\vphi \in H^1(\Omega)$ and $A \in \cO(\Omega)$.

\item \label{it:lsc} \emph{lower semicontinuity}:
$\bE(\cdot, A)$ is weakly sequentially lower semicontinuous in $H^1(\Omega)$. 
\end{enumerate}
Then $\bE$ can be represented in integral form
\begin{align*}
	\bE(\vphi,A) = \int_A f(x, \vphi, \nabla \vphi) \dd x,
\end{align*}
where the measurable function $f: \Omega \times \R \times \R^d \rightarrow [0,+\infty)$ satisfies the self-consistent formula
\begin{align}	
\label{eq:local_density}
	f(x,u,\xi)
		 := \limsup_{\eps \to 0^+} 
		 	\frac{M\big( 
					 u + \xi(\cdot - x), Q_{\eps}(x)
					\big)}{\eps^d},
\end{align}
where $Q_{\eps}(x)$ is the open cube of side-length $\eps > 0$ centred at $x$, and
\begin{align} \label{eq:def_Mpsi_neighbor}
	M(\psi, A) 
		 &:=
	\inf 
		\big\{ 
			\bE(\vphi,A) 
				\suchthat 
			\vphi \in H^1(\Omega), \; \vphi - \psi \in H^1_0(A)
		\big\}  
\end{align}
for any $\psi \in H^1(\Omega)$ and any open cube $A \subseteq \Omega$.

\end{theorem}

\begin{remark}[Equivalence of definitions] 
The paper \cite{Bouchitte2002} contains the statement of Theorem \ref{thm:representation} with $M(\psi,A)$ replaced by 
\begin{align*}
	\bar{M}(\psi, A) 
		& := \inf 
			\left\{ \bE(\vphi,A) 
				\suchthat 
				\vphi \in H^1(\Omega), \; 
				\vphi = \psi 
				\text{ in a neighbourhood of }A
			\right\}.
\end{align*}
We claim that $M = \bar{M}$. As any competitor $\vphi$ for $M$ is a competitor for $\bar{M}$, it is clear that $M \geq \bar{M}$.
To show the opposite inequality, we fix $\eps > 0$ and take $\vphi \in H^1(A)$ such that $\bE(\vphi,A) \leq \bar{M}(\psi,A) + \eps$. 
It follows that $\vphi - \psi \in H_0^1(A)$, 
and there exists a sequence 
$(\eta_n)_n \subseteq C_{\rm c}^{\infty}(A)$ such that 
$\eta_n \to \vphi - \psi$ in $H^1(\Omega)$ 
as $n \to \infty$. 
Set $\vphi_n:= \psi + \eta_n$, so that $\vphi_n \to \vphi$ in $H^1(\Omega)$. 
Note that $\vphi_n$ is a competitor for $M(\psi,A)$, 
as it coincides with $\psi$ on $A \setminus \text{spt}(\eta_n)$, 
hence 
$M(\psi,A) \leq \bE(\vphi_n,A)$ for all $n \in \N$.
Using continuity of $\bE(\cdot,A)$ with respect to the strong $H^1(\Omega)$ convergence 
(as follows from \eqref{it:Sobolev}), 
we may pass to the limit to obtain
\begin{align*}
	M(\psi,A) 
		\leq \lim_{n \to \infty} \bE(\vphi_n,A) 
		   = \bE(\vphi,A) 
		\leq \bar{M}(\psi,A) + \eps. 
\end{align*}
As $\eps > 0$ is arbitrary, the claim follows.
\end{remark}

In the remainder of this section we will verify that the functional $\bF$ from Theorem \ref{thm:mosco-pre_exist} satisfies the conditions of Theorem \ref{thm:representation}.
In particular, we will prove the locality (Section \ref{section:locality}) and the subadditivity (Section \ref{section:subadditivity}). 
The proof of Theorem \ref{thm:mosco-pre_rapp_char} will be completed at the end of Section \ref{section:subadditivity}.
The proof of Theorem \ref{thm:f_quadratic} is contained in Section \ref{sec:char}.

\subsection{Locality}	\label{section:locality}

A consequence of the inner regularity result from Proposition \ref{prop:inner_regularity} is a simple proof of the locality of $\bF$. 
An analogous result appears in \cite[Proposition 3.9]{Alicandro2004} on the cartesian grid.
The proof in our setting is much simpler due to the short range of interactions.

\begin{proposition}[Locality] \label{prop:locality}
Assume that \eqref{eq:ub} holds. 
Suppose that $\big(\tbF_N(\cdot,A)\big)_N$ is $L^2(\Omega)$-Mosco convergent to some functional $\bF(\cdot,A)$ for every $A \in \cO(\Omega)$. 
Then $\bF$ is local, i.e., for any $A \in \cO(\Omega)$ and $\vphi,\psi \in L^2(\Omega)$ such that $\vphi = \psi$ a.e.~on $A$, we have $\bF(\vphi,A) = \bF(\psi,A)$.
\end{proposition}

\begin{proof}
Let $A \in \cO(\Omega)$ and take $\vphi, \psi \in L^2(\Omega)$ such that $\vphi = \psi$ a.e.~on $A$. 
In view of the inner regularity result from Proposition \ref{prop:inner_regularity} we may assume that $\Leb^d(\partial A)=0$.
By symmetry, it suffices to prove that $\bF(\vphi,A) \geq \bF(\psi,A)$. 

Define $C_N := \bigcup\{K : K \in \cT_N \vert_A\}$ and $C := \bigcup_N C_N$, so that $C \supseteq A$. 
We claim that 
\begin{align}	\label{eq:proof_locality_inclusion}
	C \setminus A \subseteq B_N, 
	\quad \text{where } 
	 B_N := \big\{ x \in \Omega \suchthat d(x,\partial A) < 2[\cT_N] \big\}.
\end{align} 
Indeed, for every $x \in C \setminus A$ there exists $N \geq 1$ and $K \in \cT_N$ such that $x \in K \setminus A$ and $\overline{K} \cap A \neq \emptyset$. 
Therefore, $d(x,\partial A) = d(x,A) \leq \diam(K) \leq [\cT_N]$, which implies \eqref{eq:proof_locality_inclusion}.

Let $(\vphi_N)_N$ be a recovery sequence for $\bF(\vphi,A)$, i.e., $\vphi_N \to \vphi$ in $L^2(\Omega)$ and
\begin{align}\label{eq:tbF-limit}
	\lim_{N \to \infty} \tbF_N(\vphi_N, A) = \bF(\vphi,A).
\end{align}
Fix $\hat{\psi}_N \in \PCN$ such that $\hat{\psi}_N \to \psi$ in $L^2(\Omega)$ as $N \to \infty$, and define $\psi_N : \Omega \to \R$ by
\begin{align*}
	\psi_N(x) := 
	\begin{cases}
	 	 \vphi_N(x) &\text{if } x \in C , \\
	\hat{\psi}_N(x) &\text{if } x \in \Omega \setminus C .
				   \end{cases}
\end{align*}
We claim that $\psi_N \rightarrow \psi$ in $L^2(\Omega)$ as $N \to \infty$. Indeed, since $\vphi = \psi$ a.e. on $A$, we have
\begin{align}	\label{eq:proof_locality_L2bound}
\begin{aligned}
	\| \psi_N - \psi \|_{L^2(\Omega)}^2 
	&= \|  \hat{\psi}_N - \psi \|_{L^2(\Omega \setminus C)}^2 
		 + \|  \vphi_N - \psi \|_{L^2(C \setminus A)}^2
 		 + \|  \vphi_N - \vphi \|_{L^2(A)}^2. 
\end{aligned}
\end{align}
The first and the last term on the right-hand side vanish as $N \to \infty$, since $\vphi_N \rightarrow \vphi$ and $\hat{\psi}_N \rightarrow \psi$ in $L^2(\Omega)$. On the other hand, \eqref{eq:proof_locality_inclusion} yields
\begin{align*}
	\limsup_{N \to \infty} 
		\| \vphi_N - \psi \|_{L^2(C \setminus A)} 
	& \leq \limsup_{N \to \infty} 
		\big( \| \vphi \|_{L^2(B_N)} 
	   	+ \| \psi \|_{L^2(B_N)} \big)
 \\	&
  = \| \vphi \|_{L^2(\partial A)} 
 		+ \| \psi \|_{L^2(\partial A)} = 0,
\end{align*}
since $\Leb^d(\partial A) = 0$. 
Therefore, using \eqref{eq:proof_locality_L2bound} we infer that $\psi_N \rightarrow \psi$ in $L^2(\Omega)$ as $N \to \infty$.
Using this fact, the $\Gamma$-convergence of $\tbF_N$ in $L^2$, the fact that $\vphi_N = \psi_N$ a.e.~on $C$, and 
\eqref{eq:tbF-limit}, we obtain
\begin{align*}
	\bF(\psi,A) 
	\leq \limsup_{N \to \infty} \tbF_N(\psi_N,A) 
	=	 \limsup_{N \to \infty} \tbF_N(\vphi_N,A) 
	= 	 \bF(\vphi,A),
\end{align*}
which concludes the proof.
\end{proof}

\subsection{Subadditivity}	\label{section:subadditivity}

We now prove subadditivity of the functional $A \mapsto \bF(\vphi, A)$ for any $\vphi \in H^1(\Omega)$.
This is the first step towards the verification of \eqref{it:subadditivity} in Theorem \ref{thm:representation}.
The proof is inspired by \cite[Proposition 3.7]{Alicandro2004} and follows similar ideas as in the proof of Proposition \ref{prop:inner_regularity}.

\begin{proposition}[Subaddivity] \label{prop:subadditivity}
Assume \eqref{eq:ub}. 
Suppose that $\big(\tbF_N(\cdot,A)\big)_N$ is $L^2(\Omega)$-Mosco convergent to some functional $\bF(\cdot,A)$ for every $A \in \cO(\Omega)$. 
Then the functional $\bF(\vphi,\cdot)$ is subadditive for any $\vphi \in H^1(\Omega)$, in the sense that
\begin{align}	\label{eq:subadditivity_prop}
	 \bF(\vphi,A\cup B) \leq \bF(\vphi,A) + \bF(\vphi,B) \qquad \text{for all } A,B \in \cO(\Omega).
\end{align}
\end{proposition}

\begin{proof}
Fix $A,B \in \cO(\Omega)$. 
For all $A' \ssubset A$, $B' \ssubset B$, and $\vphi \in H^1(\Omega)$ we will prove that
\begin{align*}
	\bF(\vphi,A' \cup B')
		\leq \bF(\vphi,A) + \bF(\vphi,B).
\end{align*}
In view of the the inner regularity (Proposition \ref{prop:inner_regularity}), this implies  \eqref{eq:subadditivity_prop}. 

Pick $A' \ssubset A$ and $B' \ssubset B$  
and let $(\psi_N)_N$, $(\vphi_N)_N$ be recovery sequences  for $\bF(\vphi,A)$ and $\bF(\vphi,B)$ respectively, 
which we can assume to be finite. 
Fix $M \in \N$ and suppose that 
$[\cT_N] < \frac{1}{5(M+1)}$. 
We define the sets
\begin{align*}
	A_j:= \left\{ x \in A 
		\suchthat 
		d(x,A') < \frac{j}{5(M+1)} d(A', A^c) \right\} 
		\subset A
\end{align*} 
for $j \in \{1, \ldots, 5(M+1)\}$. Moreover, for $i \in \{1, ..., M \}$ let $\rho_i$ be a cutoff function $\rho_i \in C^{\infty}(\R^d)$ satisfying 
\begin{align*}
	\rho_i|_{A_{5i+2}} = 1, \quad 
	\rho_i|_{\Omega \setminus A_{5i+3}} = 0, \quad 
	0 \leq \rho_i \leq 1, \quad
	|\nabla \rho_i | \lesssim M.
\end{align*} 

We then consider the  $L^2(\Omega)$-convergent sequences
\begin{align*}
	\vphi_N^i := \Q_N P_N \Big( \rho_i \psi_N + (1-\rho_i) \vphi_N \Big) \xrightarrow[N \to \infty]{} \vphi, \quad \forall i \in \{1, \ldots, M \}.
\end{align*}
By definition, we have $\vphi_N^i \equiv \psi_N$ in $A_{5i+1}$ and $\vphi_N^i \equiv \vphi_N$ in $\Omega \setminus \overline{A_{5i+4}}$. 
Arguing as in the proof of Proposition 
\ref{prop:inner_regularity}, one deduces the bound
\begin{align}	\label{eq:proof_subadd_upbound}
	\tbF_N(\vphi_N^i,A' \cup B') 
		\leq \tbF_N(\psi_N,A) 
		+ \tbF_N\big(\vphi_N^i,(A_{5(i+1)} \setminus \overline{A_{5i}})\cap B'\big) 
		+ \tbF_N(\vphi_N,B)
\end{align}
for $i \in \{ 1, \ldots, M \}$, as well as the bound
\begin{align*}	
	\frac{1}{M} \sum_{i=1}^M 
		\tbF_N\big(\vphi_N^i,(A_{5(i+1)} \setminus \overline{A_{5i}})\cap B'\big) 
		\lesssim 
			\frac{E}{M} 
			+ \bar{k}M^2 \| \psi_N - \vphi_N \|_{L^2(\Omega)}^2,
\end{align*}
where we used that $(A_{5(i+1)} \setminus \overline{A_{5i}})\cap B' \subset A \cap B$ and that the energy of the recovery sequences $\psi_N$ and $\vphi_N$ is bounded from above, thus
\begin{align*}
	\sup_{N \in \N} 
		\tbF_N(\psi_N,A) 
	\lor \sup_{N \in \N} 
		\tbF_N(\vphi_N,B)
	 \leq 
	 	E = 
	 	E(A,B)< +\infty.
\end{align*} 
We then plug the error estimates above into \eqref{eq:proof_subadd_upbound} and deduce
\begin{align*}
	\frac{1}{M} \sum_{i=1}^M \tbF_N(\vphi_N^i,A' \cup B') -\tbF_N(\psi_N,A) - \tbF_N(\vphi_N,B) \lesssim  \frac{E}{M} + \bar{k} M^2 \| \psi_N - \vphi_N \|_{L^2(\Omega)}^2.
\end{align*}
Using the fact that $\psi_N, \vphi_N \rightarrow \vphi$ are recovery sequences, we may pass to the limit $N \to \infty$ in the previous bound and obtain, for fixed $M \in \N$,
\begin{align}	\label{eq:proof_subadd_averageupbound}
	\limsup_{N \to \infty} \frac{1}{M} \sum_{i=1}^M \tbF_N(\vphi_N^i,A' \cup B') - \bF(\vphi,A)- \bF(\vphi,B) \lesssim \frac{E}{M}.
\end{align}

Arguing again as in the proof of Proposition \ref{prop:inner_regularity}, we note that, for fixed $M \in \N$, there exists a sequence $\vphi_N^{i_N}$ satisfying $\vphi_N^{i_N} \rightarrow \vphi$ in $L^2(\Omega)$ as $N \to \infty$ and 
\begin{align*}
	\tbF_N(\vphi_N^{i_N},A' \cup B') \leq \frac{1}{M} \sum_{i=1}^M \tbF_N(\vphi_N^i,A' \cup B').
\end{align*}
Together with \eqref{eq:proof_subadd_averageupbound}, this yields
\begin{align*}
	\bF(\vphi,A' \cup B') \leq \limsup_{N \to \infty} \tbF_N(\vphi_N^{i_N},A' \cup B') \leq \bF(\vphi,A) + \bF(\vphi,B)  + C\frac{E}{M}
\end{align*}
for every $M \in \N$, for some $C=C(d,\zeta)$ and $E=E(A,B) \in \R_+$. Taking the limit $M \rightarrow \infty$, we infer that
\begin{align*}
	\bF(\vphi,A' \cup B') \leq \bF(\vphi,A) + \bF(\vphi,B)
\end{align*}
and the proof is complete.
\end{proof}

The following additivity property turns out to be much easier to prove than the corresponding result on the grid in \cite{Alicandro2004}, 
due to inner regularity in combination with the very short range of interaction (nearest neighbours on a scale of order $[\cT_N]$).

\begin{proposition}[Additivity on disjoint sets]	\label{prop:additivity}
Assume \eqref{eq:ub}. 
For any $\vphi \in H^1(\Omega)$ the function $\bF(\vphi,\cdot)$ is additive on disjoint sets, i.e.,
\begin{align}
	 \bF(\vphi,A\cup B) = \bF(\vphi,A) + \bF(\vphi,B)
\end{align}
for all $A,B \in \cO(\Omega)$ such that $A \cap B = \emptyset$.
\end{proposition}

\begin{proof}
In view of the subadditivity result from Proposition \ref{prop:subadditivity}, it remains to show superadditivity on disjoint sets.
Fix $A,B \in \cO(\Omega)$ with $A \cap B = \emptyset$. 
By inner regularity (Proposition \ref{prop:inner_regularity}) we may assume that $d(A,B) > 0$. Consequently, for $N$ sufficiently large we have
\begin{align*}
	 \tbF_N(\vphi, A \cup B) 
	 	= 
	\tbF_N(\vphi, A) + \tbF_N(\vphi, B) 
	\quad \forall \vphi \in H^1(\Omega).
\end{align*}

Fix $\vphi \in H^1(\Omega)$ and let $(\vphi_N)_N$ be a recovery sequence for $\bF(\vphi, A \cup B)$. Using the previous identity we obtain
\begin{align*}
	  		\bF(\vphi, A) 
	+ 		\bF(\vphi, B)
	 & \leq \liminf_{N \to \infty}
	 		    \tbF_N(\vphi_N, A)
		  + \liminf_{N \to \infty} 
		  		\tbF_N(\vphi_N, B)
 \\	& \leq \liminf_{N \to \infty}
	 		    \Big(
				\tbF_N(\vphi_N, A)
		  + 	\tbF_N(\vphi_N, B)
		  		\Big)
 \\	&= \liminf_{N \to \infty} 
 			\tbF_N(\vphi_N, A \cup B) 
 \\	&= \bF(\vphi, A \cup B),
\end{align*}
which is the desired superadditivity inequality.
\end{proof}

We are now in a position to collect the pieces for the proof of  Theorem \ref{thm:mosco-pre_rapp_char}.

\begin{proof}[Proof of Theorem \ref{thm:mosco-pre_rapp_char}]
In view of Theorem \ref{thm:Sob-finite}, we know that $\bF = + \infty$ outside of $H^1(\Omega)$.
To obtain the desired result on $H^1(\Omega)$ we check that $\bF(\cdot,A)$ satisfies the conditions of Theorem \ref{thm:representation}.

The locality \eqref{it:local} has been shown in Proposition \ref{prop:locality}.

To prove \eqref{it:subadditivity}, we apply the De Giorgi-Letta criterion, cf.\ \cite{degiorgi1977}, \cite{Braides1998}. For any $\vphi \in H^1(\Omega)$, it follows from Propositions \ref{prop:inner_regularity}, \ref{prop:subadditivity}, and \ref{prop:additivity} that $\bF(\vphi,\cdot)$ is the restriction of a Borel measure to $\cO(\Omega)$. 

The Sobolev upper bound \eqref{it:Sobolev} has been proved in
Proposition \ref{prop:sobolev_bound}, whereas the corresponding lower bound follows from Proposition \ref{prop:sobolev_limit}.

Finally, to prove \eqref{it:lsc} we note that lower semicontinuity with respect to strong  $L^2(\Omega)$-convergence follows from the fact any $\Gamma$-limit is lower semicontinuous; see \cite[Proposition 1.28]{braides2002}. Since $H^1(\Omega)$ is compactly embedded in $L^2(\Omega)$, the result follows.
\end{proof}

\subsection{The characterisation of the $\Gamma$-limit}\label{sec:char}
To prove Theorem \ref{thm:Mosco} it remains to characterise the $\Gamma$-limit $\bF$ obtained in Theorem \ref{thm:mosco-pre_rapp_char}. 
It thus remains to compute the function $F$ appearing in Theorem \ref{thm:mosco-pre_rapp_char}. From \eqref{eq:local_density} it follows that for $x \in \Omega$, $u \in \R$ and $\xi \in \R^d$,
\begin{align}	
\label{eq:local_density_m}
	F(x,u,\xi)
		=
	 \limsup_{\eps \rightarrow 0^+} 
	 	\frac{\bM
			\big( u + \xi(\cdot - x); Q_{\eps}(x)\big)}{\eps^d},
\end{align}
where $Q_\eps(x)$ denotes the open cube of side-length $\eps$ centred at $x$ and
\begin{align*}
	\bM(\vphi, A) := \inf_\psi 
		\big\{ \bF(\psi,A) 
				\suchthat 
			\psi \in H^1(\Omega) \text{ s.t. } 
		   \psi - \vphi \in H^1_0(A)
		\big\}
\end{align*}
for any Lipschitz function $\vphi : \Omega \to \R$ and any open set $A \subseteq \Omega$ with Lipschitz boundary.
As we will compute $\bM$ by discrete approximation, we consider its discrete counterpart $\cM_\cT$ defined by
\begin{align*}
	\cM_\cT(f, A) := \inf_g \{ \cF_\cT(g, A) \suchthat 
					g \in \R^\cT \text{ s.t. } f = g \text{ on } \cT\vert_{A^c} \}
\end{align*}
for $f : \cT \to \R$, where $\cT\vert_A$ for $A \subset \Omega$ is defined in \eqref{eq:defTA}.

\begin{remark}[Strong continuity of $\bF(\cdot,A)$ in $H^1(\Omega)$]
		\label{rem:strong_continuity}
	The quadratic nature of the discrete problems allows us to infer more information about the limit density. 
	In fact, it follows that 
	$F(x,u,\xi) = \ip{ a(x) \xi, \xi }$ 
	for some bounded matrix-valued function $a$; 
	see \cite[Remark 3.2]{Alicandro2004}. 
	Consequently, for every $A \in \cO(\Omega)$, the $\Gamma$-limit $\bF(\cdot,A)$ is continuous for the strong topology of $H^1(\Omega)$.
	This fact will be used in the proof of Lemma \ref{lem:minima-convergence} below.
\end{remark}

The following result is crucial in the proof of Theorem \ref{thm:f_quadratic}.

\begin{lemma}\label{lem:minima-convergence}
Assume \eqref{eq:ub},
and suppose that $\tbF_N(\cdot,B) \xrightarrow[]{\Gamma} \bF(\cdot,B)$ in $L^2(\Omega)$ as $N \to \infty$ for any $B \in \cO(\Omega)$. 
Then, for any $A \in \cO(\Omega)$ with Lipschitz boundary and any Lipschitz function $\vphi : \Omega \to \R$, we have 
\begin{align}	\label{eq:conv_minima}
	\cM_N(P_N \vphi,A)
	     \to
	 \bM(\vphi,A).
\end{align}
\end{lemma}

\begin{proof}
First we embed the discrete functionals in the continuous setting. 
For any Lip\-schitz function $\vphi : \bOmega \to \R$ and any open set $A \subseteq \Omega$ we set
\begin{equation}
\label{eq:def:cT_U}
	\PCN(\vphi,A) := 
	\{ \psi \in \PCN \suchthat 
	\psi(x_K) = \vphi(x_K) \ 
	\forall K \in \cT_N \vert_{A^c}  \}.
\end{equation}
We consider the embedded discrete energies $\tbF_N^\vphi : L^2(\Omega) \to [0, + \infty]$ defined by
\begin{align*}
	\tbF_N^\vphi(\psi,A)
	&:= 
	\begin{cases}
		\cF_N(P_N \psi, A)
			& \text{if }\psi \in \PCN(\vphi, A), \\
   + \infty & \text{otherwise}	,	
   \end{cases}
\end{align*}
and their continuous counterpart $\bF^\vphi: L^2(\Omega) \to [0, + \infty]$ defined by
\begin{align*}
		\bF^\vphi(\psi,A)
	&:= 
	\begin{cases}
	\bF(\psi,A)
	& \text{if }\psi - \vphi \in  H^1_0(A),  \\
	+ \infty & \text{otherwise}		.
	\end{cases}
\end{align*}
We claim that
\begin{align*}
	\tbF_N^\vphi(\cdot,A) \xrightarrow[]{\Gamma} \bF^\vphi(\cdot, A), \quad \forall A \subseteq \Omega \; \text{with Lipschitz boundary}, \; \vphi \in \Lip(\R^d),
\end{align*} 
which implies, together with Proposition \ref{prop:sobolev_limit} and by a basic result from the theory of $\Gamma$-convergence
\cite[Theorem~1.21]{braides2002}, the desired convergence of the minima in \eqref{eq:conv_minima}.
To prove the claim, we argue as in \cite[Theorem 3.10]{Alicandro2004}.

To prove the liminf inequality, we consider a sequence $\psi_N \weakto \psi$ in $L^2(\Omega)$ satisfying $\sup_N \tbF_N^\vphi(\psi_N,A) < +\infty$. 
In particular, this implies that $\psi_N \in \PCN(\vphi,A)$ and $\tbF_N^\vphi(\psi_N,A) = \tbF_N(\psi_N,A)$. 
Since $\tbF_N(\cdot,A) \xrightarrow{\Gamma} \bF(\cdot,A) $, it remains to prove that $\psi - \vphi \in H_0^1(A)$. 
In view of the boundary condition and the fact that $\vphi \in \Lip(\R^d)$, we have
\begin{align*}
		\tbF_N(\psi_N, \Omega) 
		\leq 
			\tbF_N(\psi_N, A) 
		+ 	\tbF_N(\vphi, \Omega)  
		\lesssim  
			\tbF_N(\psi_N, A) + \overline{k} \Lip(\vphi)^2.
\end{align*}
It follows from this bound and Proposition \ref{prop:sobolev_limit} that $\psi_N \rightarrow \psi$ strongly in $L^2(\Omega)$ and $\psi \in H^1(\Omega)$. 
Moreover, by construction we have $\psi_N \rightarrow \vphi$ in $L^2(\Omega \setminus A)$. 
Since $A$ has Lipschitz boundary, we conclude that $\psi - \vphi \in H^1_0(A)$.

Let us now prove the limsup inequality. 
Pick $\psi \in L^2(\Omega)$ such that $\bF^\vphi(\psi,A) < +\infty$. 
In particular, $\psi - \vphi \in H_0^1(A)$. 
Without loss of generality, 
we may assume that $\supp(\psi - \vphi) \ssubset A$, 
as the general case follows from this by a density argument using the continuity of $\bF$ 
in the strong $H^1(\Omega)$-topology; 
see Remark \ref{rem:strong_continuity}. 
Consider a recovery sequence 
$\psi_N \to \psi$ in $L^2(\Omega)$ such that 
$\tbF_N(\psi_N,A) \to \bF(\psi,A) = \bF^\vphi(\psi,A)$ 
as $N \to \infty$.  
Now we argue as in the proof of Proposition \ref{prop:inner_regularity}. 
For any $\delta > 0$ there exists a cutoff function $\zeta_\delta$ with the following properties: 
\begin{enumerate}[(i)]
\item 
	$\supp(\psi - \vphi) 
		\ssubset \supp \zeta_\delta 
		\ssubset A$;
\item the functions 
	$\psi_N^\delta
	 	 := Q_N \circ P_N
			\big( 
					    \zeta_\delta \psi_N
				   + (1-\zeta_\delta) \vphi 
			\big)$
satisfy
\begin{align*}
		\limsup_{N \to \infty} 
			\tbF_N^\vphi(\psi_N^\delta,A)
	& = \limsup_{N \to \infty} 
			\tbF_N(\psi_N^\delta,A) 
 \\& \leq \limsup_{N \to \infty} 
  			\tbF_N(\psi_N,A) + \delta 
	= 	\bF^\vphi(\psi,A) + \delta.
\end{align*}
\end{enumerate}
Passing to the limit $\delta \to 0$ using a diagonal subsequence $\psi_N^{\delta(N)} \to \psi$ in $L^2(\Omega)$, the result follows.
\end{proof}

\begin{proof}[Proof of Theorem \ref{thm:f_quadratic}]
We split the proof into two parts. 

\medskip

\emph{Step 1.} 
We first suppose that 
$\mu$ is the normalised Lebesgue measure and $m_N(K) = \pi_N(K) =  \frac{|K|}{|\Omega|}$, and we fix $\eps > 0$. 
For fixed $b \in \R$, $z \in \Omega$, and $\xi \in \R^d$ we will compute 
\begin{align*}
	\cM_N\big( f_N, Q_\eps (z) \big), 
	\quad \text{where } 
		f_N(K) 
			:= \vphi_{b,z}^\xi(x_K)
	\text{ and } 
		\vphi_{b,z}^\xi(\cdot) 
			:= u + \xi(\cdot - z)
\end{align*}
As a shorthand we write
 $Q_\eps := Q_\eps(z)$.
Recall that
\begin{align*}
	\cM_N(f, Q_\eps) 
	& = \inf_g \big\{ \cF_N(g,Q_\eps) 
			\suchthat g \in \R^{\cT_N} \text{ and }
		g(K) = f(K) 
	\text{ for } K \in \cT_N\vert_{Q_\eps^c}
	\big\}.
\end{align*}
In other words, we minimise the discrete Dirichlet energy localised on $Q_\eps$ with Dirichlet boundary conditions given by the discretised affine function $f$. 
As follows by computing the first variation of the action, the unique minimiser is given by the solution $h : \cT_N \to \R$ of the corresponding discrete Laplace equation 
\begin{align}
\label{eq:discreteEL}
\begin{cases} 
	\cL_N h(K) = 0 
		& \text{for } K \in \cT_N \setminus \cT_N \vert_{Q_\eps^c}, \\ 
	h(K) = f_N(K)
		& \text{for } K \in \cT_N \vert_{Q_\eps^c}.
\end{cases}		
\end{align}
We claim that the function $f_N$ solves \eqref{eq:discreteEL}. 
Indeed, the boundary conditions hold trivially. 
Moreover, writing $\tau_{KL} := \frac{x_K - x_L}{|x_K - x_L|}$ we obtain for any $K \in \cT_N \setminus \cT_N \vert_{Q_\eps^c}$,
\begin{align*}
	\pi_N(K) \cL_N f_N(K) 
	  & = \sum_{L \sim K} \frac{|\Gamma_{KL}|}{d_{KL}} 
	  		\big(f_N(L) - f_N(K)\big)
		= - \sum_{L\sim K} |\Gamma_{KL}| \ip{\xi, \tau_{KL}}
	\\&	= \int\limits_{\partial K}\ip{ \xi, \nu_{\text{ext}} } 
			\dd \Haus^{d-1} 
		= 0,
\end{align*}
where $\nu_{\text{ext}}$ denotes the outward normal unit normal and in the last step we used Stokes' theorem. 
This computation shows the optimality of $f$ and hence
\begin{align*}
	\cM_N(f_N, Q_\eps) = \cF_N(f_N, Q_\eps).
\end{align*}

For the asymptotic computation of $\cF_N(f_N, Q_\eps)$
we use the average isotropy property of any regular mesh (see \cite[Lemma 5.4]{gladbach2018}) to obtain 
\begin{align*}
	\left| \cF_N(f_N, Q_\eps) - \eps^d |\xi|^2 \right| 
	& = \bigg|\bigg( \frac14
		 \sum_{\substack{K,L \in \cT_N 
	\\ \mathclap{\overline{K},\overline{L} \cap Q_\eps \neq \emptyset}}} d_{KL} 
		|\Gamma_{KL}| \ip{ \xi, \tau_{KL} }^2  \bigg)
		  - |\xi|^2 |Q_\eps| \bigg| 
\\
	& \leq \big|B\big(\partial Q_\eps, 5[\cT_N]\big)\big| \rightarrow 0 \text{ as }N \to \infty,
\end{align*} 
where $B(C,r) := \{x \in \Omega \suchthat d(x,C) < r\}$.
Note that we get $|B(\partial Q_\eps, 5[\cT_N])|$ instead of $|B(\partial Q_\eps, 4[\cT_N])|$ as in \cite[Lemma 5.4]{gladbach2018} because we take into account all the cells whose closure intersects the cube $Q_\eps$ and not only the ones contained in it. 
Together with Lemma \ref{lem:minima-convergence}, we obtain, for all $\xi \in \R^d$ and $\eps > 0$,
\begin{align}	
\label{eq:limit_M_1}
	\bM(\vphi_{b,z}^\xi, Q_\eps) 
		= \lim_{N \to \infty} 
			\cM_N(f, Q_\eps)
		= \lim_{N \to \infty} 
			\cF_N(f, Q_\eps)
		= \eps^d |\xi|^2, 
\end{align}
hence
\begin{align*}
	F(x,u,\xi) 
		= \limsup_{\eps \to 0^+} 
			\frac{\bM(\vphi_{b,z}^\xi, Q_\eps)}{\eps^d} 
		= |\xi|^2,
\end{align*}
which concludes the proof in the special case $\sigma$, $\rho \equiv 1$, $m_N = \pi_N$. 

\medskip

\emph{Step 2.} Let us now consider the general case where 
$m_N$ and $\mu$ satisfy \eqref{eq:lb}, \eqref{eq:ub}, and \eqref{eq:pc}.
We write $\bar \cF_N, \bar \cM_N$ for the analogues of $\cF_N, \cM_N$ in the special case where $\mu$ is the normalised Lebesgue measure and $m_N = \pi_N$, which we considered in Step 1. 

Fix $b \in \R$, $z \in \Omega$, and $\xi \in \R^d$, and let $Q_\eps$, $\vphi_{b,z}^\xi$, and $f$ be as above.
Furthermore, let $\upsilon_N$ be the density of $\mathsf{Q}_N m_N$ with respect to the Lebesgue measure. 
For all $g : \cT_N \to \R$ we have by construction,
\begin{align*}
 	\Big( \inf_{Q_{2\eps}} \upsilon_N \Big)
		 \bar\cF_N(g, Q_\eps) 
	\leq \cF_N(g, Q_\eps) 
	\leq \Big( \sup_{Q_{2\eps}  } \upsilon_N \Big)
		 \bar\cF_N(g, Q_\eps),
\end{align*}
hence, in particular,
\begin{align*}
	\Big( \inf_{Q_{2\eps}} \upsilon_N  \Big) 
			\bar\cM_N(f, Q_\eps) 
	\leq \cM_N(f, Q_\eps) 
	\leq \Big( \sup_{Q_{2\eps}} \upsilon_N \Big) 
			\bar\cM_N(f, Q_\eps).
\end{align*}
As a consequence of the first part of the proof and \eqref{eq:limit_M_1}, taking the limit as $N \to \infty$ and applying \eqref{eq:conv_minima}, we deduce
\begin{align*}
	\Big( \limsup_{N \to \infty} 
			\inf_{Q_{2\eps}} \rho_N \Big)  
					|\xi|^2 \eps^d 
	   & \leq \bM(\vphi_{b,z}^\xi, Q_\eps) 
	\leq \Big( \liminf_{N \to \infty} 
		\sup_{Q_{2\eps} } \rho_N \Big) 
		|\xi|^2 \eps^d.
\end{align*}
Taking the limsup as $\eps \to 0$, 
we deduce from \eqref{eq:local_density_m}
and the condition \eqref{eq:pc},  
\begin{align*}
	F(x,u,\xi)
	 = \limsup_{\eps \to 0} 
	 	\frac{ \bM(\vphi_{b,z}^\xi, Q_\eps) }{\eps^d} 
	 = |\xi|^2 \upsilon(x) 
	 \quad \text{ for a.e. }z \in \Omega,
\end{align*}
which concludes the proof.
\end{proof}

\subsection*{Acknowledgment}
This work is supported by the European Research Council (ERC) under the European Union's Horizon 2020 research and innovation programme (grant agreement No 716117) and by the Austrian Science Fund (FWF), grants No F65 and W1245.

\bibliographystyle{my_alpha}
\bibliography{literatureEV}

% \bibliographystyle{plain}
% \bibliography{literatureEV}

\end{document}